\documentclass[reqno]{amsart}

\usepackage{stmaryrd}
\SetSymbolFont{stmry}{bold}{U}{stmry}{m}{n}
\usepackage{color}
\usepackage[dvipsnames]{xcolor}
\usepackage{ifpdf}
\ifpdf 
\usepackage[pdftex]{graphicx}   % to include graphics
\usepackage[pdftex,     % sets up hyperref to use pdftex driver
plainpages=false,   % allows page i and 1 to exist in the same document
breaklinks=true,    % link texts can be broken at the end of line
colorlinks=true,
linkcolor=red,
citecolor=green,
pdftitle=My Document
pdfauthor=My Good Self
]{hyperref} 
\else 
\usepackage{graphicx}       % to include graphics
\usepackage{hyperref}       % to simplify the use of \href
\fi 

\hypersetup{colorlinks, linkcolor=red!50!black, citecolor=green!50!black, urlcolor=[rgb]{0.3, 0.2, 0.5}}

\usepackage{subfig}
\usepackage{glossaries}
\usepackage{glossary-mcols}

\usepackage{aurical}
\usepackage{amsfonts,amsmath}	
\usepackage{amssymb}
\usepackage{verbatim}
\usepackage{amsopn}
\usepackage[english]{babel}
\usepackage{amsthm}
\usepackage{enumerate}
\usepackage{mathrsfs}	
\usepackage{mathtools}
\usepackage{esint}
\usepackage{todonotes}
\usepackage{enumitem}

\usepackage{bm}
\usepackage{bbm}
\usepackage{caption}

\usepackage{graphicx}% http://ctan.org/pkg/graphicx
\usepackage{subfig}

\usetikzlibrary{intersections,calc}

\date{\today}

\theoremstyle{definition} \newtheorem{definition}{Definition}[section]
\theoremstyle{definition} \newtheorem{remark}[definition]{Remark}
\theoremstyle{plain} \newtheorem{lemma}[definition]{Lemma}
\theoremstyle{plain} \newtheorem{proposition}[definition]{Proposition}
\theoremstyle{plain} \newtheorem{theorem}[definition]{Theorem}
\theoremstyle{plain} \newtheorem{corollary}[definition]{Corollary}
\theoremstyle{definition} 
\theoremstyle{plain} 
\theoremstyle{definition} 
\theoremstyle{plain}

\DeclareMathOperator{\sign}{sign}
\DeclareMathOperator{\BV}{BV}
\DeclareMathOperator{\FV}{FV}

\let\d\relax
\newcommand{\d}{\partial}

\DeclareMathOperator{\dive}{div}

\DeclareMathOperator{\supp}{supp}

\DeclareMathOperator{\Lip}{Lip}

\DeclareMathOperator{\Tan}{{\mathsf {Tan}}}
\DeclareMathOperator{\ext}{{\mathsf {ext}}}

\newcommand{\R}{\mathbb{R}}

\newcommand{\N}{\mathbb{N}}

\newcommand{\e}{\varepsilon}
\newcommand{\eps}{\varepsilon}
\newcommand{\fhi}{\varphi}

\newcommand{\loc}{\text{\rm loc}}

\newcommand{\M}{\bm{M}}

\renewcommand{\L}{\mathscr L}

\newcommand{\1}{\mathbbm 1}

\newcommand{\meas}[1]{\left|#1\right|}

\usepackage{nicefrac}
\usepackage{faktor}

\renewcommand{\L}{\mathscr L}
\renewcommand{\H}{\mathscr H}
\newcommand{\rest}{\llcorner}
\renewcommand{\M}{\mathscr M}

\numberwithin{equation}{section}

\theoremstyle{plain} \newtheorem*{theorem*}{Theorem}
\theoremstyle{plain} 
\theoremstyle{plain} \newtheorem*{mthm*}{Main Theorem}
\theoremstyle{plain} \newtheorem*{conjecture*}{Conjecture}
\theoremstyle{plain} 
\theoremstyle{plain} \newtheorem*{problem*}{Problem}

\title[]{On the structure of divergence-free measures on $\R^2$}

\author{Paolo Bonicatto}
\address{Mathematics Institute, University of Warwick, Coventry CV4 7AL, UK}
\email{Paolo.Bonicatto@warwick.ac.uk}

\author{Nikolay A. Gusev}
\address{Moscow Institute of Physics and Technology,
	9 Institutskiy per., Dolgoprudny, Moscow Region, 141700;
	Steklov Mathematical Institute of Russian Academy of Sciences,
	8 Gubkina St, Moscow, 119991;}
\email{n.a.gusev@gmail.com}

\begin{document}
\begin{abstract}
We consider the structure of divergence-free vector measures on the plane. We show that such measures can be decomposed into measures induced by closed simple curves. More generally, we show that if the divergence of a planar vector-valued measure is a signed measure, then the vector-valued measure can be decomposed into measures induced by simple curves (not necessarily closed). As an application we generalize certain rigidity properties of divergence-free vector fields to vector-valued measures. Namely, we show that if a locally finite vector-valued measure has zero divergence, vanishes in the lower half-space and the normal component of the unit tangent vector of the measure is bounded from below (in the upper half-plane), then the measure is identically zero.\\



\end{abstract}

	\maketitle

	\tableofcontents
	
\section{Introduction} 
In this paper we study the structure of vector-valued Borel measures $\bm \mu$ solving the equation 
\begin{equation}\label{e-div-mu-eq-rho}
\dive \bm \mu = \rho
\end{equation}
in the sense of distribution on $\R^d$, where $d \ge 1$ and $\rho$ is a given ($\R$-valued) Borel measure on $\R^d$. Many equations of the mathematical physics can be written in the form of \eqref{e-div-mu-eq-rho}, for instance the continuity equation.
A simple (but important) example of a measure satisfying \eqref{e-div-mu-eq-rho} is a \emph{measure $\bm \mu_\gamma$ induced by a Lipschitz curve $\gamma\colon [0,1]\to \R^d$}, which is defined (via Riesz-Markov-Kakutani Theorem) by 
\begin{equation}\label{e-mu-gamma}
\langle \bm \mu_\gamma, \Phi \rangle \equiv \int_{\R^d} \Phi \cdot \, d \bm \mu_\gamma := \int_0^1 \Phi(\gamma(t)) \cdot \gamma'(t) \, dt\qquad \forall \Phi \in C_0(\R^d;\R^d).
\end{equation}
Here $C_0(\R^d;\R^d)$ is the closure of the set of compactly supported continuous functions $C_c(\R^d;\R^d)$ with respect to the uniform norm. It is easy to see that $\bm \mu_\gamma$ solves \eqref{e-div-mu-eq-rho} with $\rho := \delta_{\gamma(0)} - \delta_{\gamma(1)}$,
where $\delta_x$ with $x\in \R^d$ denotes the Dirac measure concentrated at $x$. 

Clearly every finite linear combination of measures of the form $\bm \mu_\gamma$ still solves \eqref{e-div-mu-eq-rho}.
More generally, 
let $\Gamma:= \Lip([0,1];\R^d)$ denote the space of all Lipschitz functions $f\colon [0,1]\to \R^d$, endowed with the $\sup$-norm.
Let $\mathscr M(X;\R^d)$ denote the set of $\R^d$-valued Borel measures on a topological space $X$ (for $d=1$ we will simply write $\mathscr M(X) := \mathscr M(X;\R)$ and $\mathscr M_+(X)$ for the set of \emph{non-negative} Borel measures). Let $|\bm \mu|$ denote the total variation of $\bm \mu \in \mathscr M(\R^d;\R^d)$ and recall that $\|\bm \mu\|:= |\bm \mu|(\R^d)$ is a norm on $\mathscr M(\R^d;\R^d)$.
Suppose that $\eta \in \mathscr M_+(\Gamma)$ is such that $\int_\Gamma \|\bm \mu_\gamma\| \, d\eta(\gamma) < \infty$.
Then using Fubini's Theorem one can show that the measure \begin{equation*}
\bm \mu := \int_{\Gamma} \mu_\gamma \, d\eta(\gamma),
\end{equation*}
which is defined by
\begin{equation*}
\langle \bm \mu, \Phi \rangle := \int_{\Gamma} \langle \bm \mu_\gamma, \Phi \rangle \, d\eta(\gamma) \qquad \forall \Phi \in C_0(\R^d;\R^d),
\end{equation*}
solves \eqref{e-div-mu-eq-rho} with $\rho:= \int_{\Gamma} (\delta_{\gamma(0)} - \delta_{\gamma(1)})\, d\eta(\gamma)$ (which is defined similarly).
Therefore a natural question is whether the converse implication holds true, i.e. if any solution $\bm \mu \in \mathscr M(\R^d;\R^d)$ of \eqref{e-div-mu-eq-rho} (with some $\rho \in \mathscr M(\R^d)$) can be written as $\int_{\Gamma} \mu_\gamma \, d\eta(\gamma)$ for some $\eta \in \mathscr M_+(\Gamma)$. 

Decompositions of this kind were used in \cite{ST2017} in order to derive the so-called superposition principle for the measure-valued solutions of the continuity equation
(which was proved in \cite[Thm. 12]{umibologna} for Euclidean spaces).
In turn, such a superposition principle was used in \cite{BBG2016} in order to obtain certain uniqueness results for solutions of the continuity equation. The main result of the present paper can be stated as follows: 
 \begin{mthm*}\label{t-main}
Let $d=2$.
Suppose that $\rho \in \mathscr M(\R^d)$ and $\bm \mu \in \mathscr M(\R^d;\R^d)$ solve \eqref{e-div-mu-eq-rho}. Then there exists $\eta \in \mathscr M_+(\Gamma)$ such that
\begin{subequations}
\begin{equation}\label{e-mu-decomp}
\bm \mu = \int_{\Gamma} \bm \mu_\gamma \, d\eta(\gamma)
\end{equation}
\begin{equation}\label{e-|mu|-decomp}
|\bm \mu| = \int_{\Gamma} |\bm \mu_\gamma| \, d\eta(\gamma)
\end{equation}
\begin{equation}\label{e-|div-mu|-decomp}
|\dive \bm \mu| = \int_{\Gamma} |\dive \bm \mu_\gamma| \, d\eta(\gamma)
\end{equation}
\end{subequations}
and for $\eta$-a.e. $\gamma \in \Gamma$ there exists $\tilde\gamma\in \Gamma$ which is injective on $[0,1)$ such that $\bm \mu_\gamma = \bm \mu_{\tilde \gamma}$.
\end{mthm*}

For $d>2$ in general such a decomposition is not possible due to examples provided in the celebrated paper \cite{Smi93} (in particular one can consider $\bm \mu$ associated with an irrational winding of a torus).
However in \cite{Smi93} it was proved that for any $d>2$ the measure $\bm \mu$ can be decomposed into the so-called \emph{elementary solenoids} in such a way that \eqref{e-mu-decomp}--\eqref{e-|div-mu|-decomp} hold.
Recently this decomposition result was generalized for metric spaces in \cite{PS2012,PS2013}.
Note that for $d>2$ the set of elementary solenoids is strictly larger than the set of measures induced by Lipschitz curves. However, by \hyperref[t-main]{Main Theorem}, all elementary solenoids are induced by Lipschitz curves in the case $d=2$.

Following \cite{PS2012} $\bm \sigma \in \mathscr M(\R^d;\R^d)$ will be called \emph{a cycle of} $\bm \mu$ if $\dive \bm \sigma = 0$ and $\|\bm \mu\| = \|\bm \mu - \bm \sigma\| + \|\bm \sigma\|$.
Moreover $\bm \mu$ will be called \emph{acyclic} is $\bm \sigma=0$ is the only cycle of $\bm \mu$.
It is known \cite{Smi93,PS2012} that any measure $\bm \mu \in \mathscr M(\R^d;\R^d)$ can be decomposed into cyclic and acyclic parts (see e.g. Proposition 3.8 in \cite{PS2012}):

\begin{theorem}\label{t-maximal-subcycle}
For any $\bm \mu \in \mathscr M(\R^d;\R^d)$
there exists a cycle $\bm \sigma$ of $\bm \mu$ such that $\bm \mu - \bm \sigma$ is acyclic.
\end{theorem}

A curve $\gamma \in \Gamma$ will be called \emph{simple} if $\gamma$ is injective on $[0,1)$.
The acyclic part of $\bm \mu$ solving  \eqref{e-div-mu-eq-rho} can be decomposed into measures induced by simple Lipschitz curves (see e.g. Theorem 5.1 in \cite{PS2012}):
\begin{theorem}\label{t-acyclyc-decomposition}
If %$\bm \mu$ solves \eqref{e-div-mu-eq-rho} and 
$\bm \mu$ is acyclic then there exists $\eta \in \mathscr M_+(\Gamma)$ such that \eqref{e-mu-decomp}--\eqref{e-|div-mu|-decomp} hold
and for $\eta$-a.e. $\gamma \in \Gamma$ there exists simple $\tilde\gamma\in \Gamma$ such that $\bm \mu_\gamma = \bm \mu_{\tilde \gamma}$.
\end{theorem}

In view of Theorems~\ref{t-maximal-subcycle} and~\ref{t-acyclyc-decomposition} it is sufficient to prove \hyperref[t-main]{Main Theorem} for $\rho=0$. We provide two different proofs of this result.
Both proofs are based on a weak version of Poincar\'e Lemma: every divergence-free measure $\bm \mu$ in $\R^2$ can be represented as $\bm \mu =\nabla^\perp f$, where $\nabla^\perp = (-\d_2, \d_1)$ and $f \colon \R^2 \to \R$ is a locally integrable ``potential'' function with finite total variation.

The first proof (inspired by a remark in \cite{Smi93}) exploits functional analytic tools and relies on  Choquet's Theorem (see e.g. \cite{phelps}), in view of which it suffices to characterize the extreme points of the unit ball in the space of divergence-free measures. Using the weak version of Poincar\'e Lemma mentioned above we construct a certain space of functions with finite total variation, denoted by $\FV(\R^2)$, which is isometrically isomorphic (via the mapping $\nabla^\perp$) to the space of divergence-free measures. Then it remains to characterize the extreme points of the unit ball in $\FV(\R^2)$. In order to do this we apply the Coarea Formula and a fine analysis of sets of finite perimeter using the techniques from \cite{ACMM}. Eventually we show that the extreme points of the unit ball in $\FV(\R^2)$ are (normalized) characteristic functions of \emph{simple} sets (see Definition \ref{def:indecomposable} and Definition \ref{def:simple}). Using the results from \cite{ACMM} and \cite{BG}, we show that the divergence-free measures associated to extreme points are induced by closed simple Lipschitz curves.

In the second proof of \hyperref[t-main]{Main Theorem} we construct the appropriate measure $\eta$ directly. First we decompose the ``potential'' $f$ of $\bm \mu$ into a countable family of \emph{monotone} functions $f_k\in \FV(\R^2)$ using a modification of the result from \cite{BT} (which we prove in the Appendix).
Then we construct the desired measure $\eta_k$ for each component $f_k$ directly using the Coarea Formula and ultimately construct $\eta$ as the sum of $\eta_k$. An advantage of this approach is that it provides a more detailed description of the measure $\eta$ in view of monotonicity of $f_k$.

\subsection{Applications to rigidity properties of vector-measures} As an application of our decomposition of vector-measures into measures induced by curves, we establish a certain \emph{rigidity property} for vector-valued measures (extending one of the results from \cite{LeoSar17}). Let $\mathscr M_\loc(\R^d)$ denote the space of locally finite Borel measures on $\R^d$. Rigidity properties were introduced in the paper \cite{LeoSar17} to study fine properties of the \emph{trace} (in the Anzellotti's sense \cite{Anzellotti1983}) of bounded, divergence-free vector fields on a class of rectifiable sets. 
Here we consider the following generalization of Definition~1.1 from \cite{LeoSar17}: recall that, given $\bm \mu \in \mathscr M_\loc(\R^d;\R^d)$, by polar decomposition (see e.g. \cite{AFP}, Corollary 1.29) there exists a unique $\tau \in L^1_\loc(|\bm \mu|;\R^d)$ with $|\tau(x)| = 1$ for $|\bm \mu|$-a.e. $x\in \R^d$ such that $\bm \mu = \tau |\bm \mu|$. 

\begin{definition}\label{d-linear-rigidity}
  Let $\mathscr F \subset \mathscr M_\loc(\R^d;\R^d)$.
	We say that the \emph{linear rigidity property holds for $\mathscr F$} if for any $c>0$ and for any $\bm \nu\in \mathscr F$ such that 
	\begin{enumerate}
		\item[(i)] $\bm \nu\left(\left\{ x=(x_1, \ldots, x_d) \in \R^d: x_d \le 0 \right\}\right) = 0$;
		\item[(ii)] $\dive \bm \nu = 0$ in the distributional sense; 
		\item[(iii)] $\tau_d(x) \ge c |\tau(x)|$ for $|\bm \nu|$-a.e. $x \in \R^d$;
	\end{enumerate} 
one has that $\bm \nu =0$.
\end{definition}

For $\mathscr F$ consising of locally finite vector measures which are absolutely continuous with respect to Lebesgue measure (and have uniformly bounded density) the linear rigidity property was established in~\cite{LeoSar17}, Theorem~1.2.
Using the decomposition of vector measure into measures induced by curves we can prove the following result, which holds true in \emph{every} dimension:
\begin{theorem}\label{t-linear-rigidity}
For any $d\in \N$, the linear rigidity property holds for $\mathscr F = \mathscr M_\loc(\R^d;\R^d)$.
\end{theorem}

\section{Preliminaries and notation} 

In this section, we collect some useful and preliminary results and we set some notations that will be used throughout the paper. 

\subsection{General notation}
The $d$-dimensional Euclidean space will be denoted by $\R^d$, with $d \ge 1$. Usually, $\Omega \subset \R^d$ stands for a generic open set. The indicator (also characteristic) function of a set $A$ is denoted by $\1_A$ and the complement by $A^c$. The Lebesgue measure on $\R^d$ will be $\L^d$ while the $k$-dimensional Hausdorff measure, for $k \le d$, will be $\mathscr H^{k}$. If $(X, \Vert \cdot \Vert)$ is a normed space we will denote by $B_1^X$ the closed unit ball with center $0$ and radius $1$, i.e. 
\begin{equation*}
B_1^X := \{x \in X: \Vert x \Vert \le 1\}. 
\end{equation*}
If $U \subset X$, the notation $\bar U$ will denote the closure of $U$.
\noindent

If $\mu$ is a measure, the \emph{restriction of $\mu$} to some measurable subset $A$ is $\mu \rest_A$. The space of $p$-integrable functions (resp. locally $p$-integrable functions) on $\Omega$ will be denoted in the usual way by $L^p(\Omega)$ (resp. $L^p_\loc(\Omega)$), for $1\le p \le +\infty$, and the symbol $\Vert \cdot \Vert_p$ will stand for the usual norm in the former space. 

If $X$ is a topological space, $\mathscr M(X;\R^d)$ will denote the set of $\R^d$-valued Borel measures on~$X$. 
For $d=1$ let $\mathscr M(X) := \mathscr M(X;\R)$ and let $\mathscr M_+(X)$ denote the set of $[0,+\infty]$-valued Borel measures. For any~$\bm \mu \in \mathscr M(X;\R^d)$ let $|\bm \mu| \in \mathscr M_+(X)$ denote the associated total variation measure. Recall that
\begin{equation*}
\| \bm \mu \|_{\mathscr M} := |\bm \mu| (X)
\end{equation*}
is a norm on $\mathscr M(X; \R^d)$ with respect to which this space is complete (see e.g. \cite{bogachev}).

If $X$ is a locally compact and separable metric space then $\mathscr M(X;\R^d)$ can be identified (by Riesz-Markov-Kakutani Theorem) with the dual of $C_0(\R^d;\R^d)$, where $C_0(\R^d;\R^d)$ is the closure of the set of compactly supported continuous functions $C_c(\R^d;\R^d)$ with respect to the uniform norm. By default in this case we will endow $\mathscr M(X;\R^d)$ with the weak* topology. Note that the total variation norm on $\mathscr M(X;\R^d)$ coincides with the norm induced by duality with $C_0(X;\R^d)$ (see e.g. \cite[Thm. 1.54]{AFP}).

Recall also the definition of push-forward of a measure $\mu$ on some space $X$ through a Borel map $f \colon X \to Y$: we denote by $f_{\#} \mu$ the measure on $Y$ defined by $(f_{\#} \mu )(A):= \mu(f^{-1}(A))$ for any Borel set $A \subset Y$. It is well known that the measure $f_{\#} \mu$ satisfies the following equality for every bounded Borel function $\phi \colon Y \to \R$
\begin{equation}\label{eq:push_forward}
\int_X\phi(f(x)) \, d\mu(x) = \int_Y \phi(y)\,  d (f_{\#} \mu)(y). 
\end{equation}

\noindent
The \emph{divergence} of a vector-valued measure $\bm \mu \in \mathscr M(\R^d; \R^d)$ is understood in the sense of distributions, i.e. 
\begin{equation*}
\langle \dive \bm \mu, \phi \rangle := -\int_{\R^d} \nabla \phi(x) \, d \bm \mu(x) \qquad \forall  \phi \in C_c^\infty(\R^d).
\end{equation*}

\subsection{$\BV$ functions, perimeters, tangents} 
Let $\Omega \subset \R^d$ be an open set. 
\begin{definition}[$\BV$ functions, {\cite[Definition 3.1]{AFP}}] We say that a function $u \in L^1(\Omega)$ has \emph{bounded variation in $\Omega$} if the distributional derivative of $u$ is representable by a finite Radon measure in $\Omega$, i.e. 
	\begin{equation*}
	\int_{\Omega} u \frac{\partial \phi}{\partial x_i} \, dx = - \int_{\Omega} \phi d(D_iu) \quad \text{ for every } \phi \in C_c^\infty(\Omega) \, \text{ and for every } i= 1,\ldots, d 
	\end{equation*}
	for some $\R^d$-valued vector measure $(D_1u, \ldots, D_du)$ in $\Omega$. The space of functions of bounded variation in $\Omega$ is denoted by $\BV(\Omega)$.
\end{definition}

The space $\BV(\Omega)$ is a normed space under the norm 
\begin{equation*}
\Vert u \Vert_{\BV}:= \Vert u \Vert_{1} + \Vert D u \Vert_{\mathscr M}.
\end{equation*}

\begin{definition}[Variation, {\cite[Definition 3.4]{AFP}}]\label{def:variation} Let $u \in L^1_\loc(\Omega)$. The \emph{variation} $V(u, \Omega)$ of $u$ in $\Omega$ is defined by 
	\begin{equation*}
	V(u,\Omega):= \sup \left\{ \int_{\Omega} u(x) \dive\bm \phi(x) \, dx: \, \bm \phi \in C_c^\infty(\Omega;\R^d), \, \Vert \bm \phi \Vert_{\infty}\le 1 \right\}.
	\end{equation*}
\end{definition}

The variation enjoys several properties (see e.g. \cite[Remark 3.5]{AFP}): the map $u \mapsto V(u,\Omega)$ is l.s.c. in the $L^1_\loc(\Omega)$-topology. On the other hand, for fixed $u \in L^1_\loc(\Omega)$, it is possible to define $V(u,A)$ for any open set $A \subset \Omega$ and then, via the Cartheodory construction, extend $V(u,\cdot)$ to a Borel measure that will still be denoted by $V(u,\cdot)$.
Such measure has finite total variation in $\Omega$ if and only if $u \in \BV(\Omega)$ and in this case $V(u,\Omega)=\vert Du \vert(\Omega)$ (see \cite[Proposition 3.6]{AFP}). For simplicity, we will simply write $V(u)$ to denote the variation in the full space $V(u,\R^d)$. 

We recall that, as for Sobolev spaces, $\BV$ functions enjoy some higher  integrability properties: these are usually expressed via embedding theorems. For our purposes, the following general result will be needed. 

\begin{theorem}[$\BV$ embeddings, {\cite[Theorem 3.47]{AFP}}]\label{thm:347AFP}
	Let $d \ge 1$. For any function $u \in L^1_\loc(\R^d)$ satisfying $V(u) < \infty$ there exists a unique constant $m \in\R$ such that
	\begin{equation*}
	\Vert u - m \Vert_{L^{1^*}(\R^d)} \le \gamma V(u) 
	\end{equation*}
	for some universal constant $\gamma = \gamma(d)$, where
	\begin{equation}\label{eq:sobolev_conj}
	1^* := \begin{cases}
	\frac{d}{d-1}, & d>1 \\
	\infty, & d=1.
	\end{cases}
	\end{equation} 
	If $u \in L^1(\R^d)$ then $m = 0$, $u \in \BV(\R^d)$ and hence 	$\Vert u \Vert_{L^{1^*}(\R^d)} \le \gamma V(u)$. In particular, the embedding $\BV(\R^d) \hookrightarrow L^{1^*}(\R^d)$ is continuous.
\end{theorem}

\begin{definition}[Sets of finite perimeter, {\cite[Definition 3.35]{AFP}}] A Lebesgue measurable set $E \subset \R^d$ is said to be of \emph{finite perimeter in} $\Omega \subset \R^d$ if the variation of $\1_E$ in $\Omega$ is finite and the perimeter of $E$  in $\Omega$ is 
	\begin{equation*}
	P(E,\Omega):= V(\1_E, \Omega). 
	\end{equation*}
\end{definition}
We recall also the \emph{Coarea Formula} for functions of bounded variation (see \cite[Theorem 3.40]{AFP}): for any function $u \in L^1(\Omega)$ it holds  
\begin{equation*}
V(u,\Omega) := \int_{\R} P(\{x \in \Omega: u(x)>t\}) \, dt
\end{equation*}
where the equality is understood in the sense that RHS is finite if and only if the LHS is finite and in this case their values coincide and $u \in \BV(\Omega)$. 

We also recall the \emph{isoperimetric inequality}, see \cite[Theorem 3.46]{AFP}: for any set $E\subset \R^d$, $d>1$, of finite perimeter either $E$ or $\R^d\setminus E$ has finite Lebesgue measure and there exists a universal constant $c=c(d)$ such that 
\begin{equation}\label{eq:isoperimetry}
\min\{\L^d(E), \L^d(\R^d \setminus E)\} \le c(d) P(E)^{\nicefrac{d}{d-1}}
\end{equation}

We now recall the definition of \emph{approximate tangent space} to a rectifiable set. Let $k \in \N$ with $k \ge 1$. If $\mu$ is a Radon measure on $\R^d$ and $E \subset \R^d$ is a Borel we define, for $x \in \R^d$ and $r>0$, 
\begin{equation*}
\mu_{x,r}(E) := (\Phi_{x,r})_{\#} \mu(E), \qquad \text{ where } \qquad \Phi_{x,r}(y) := \frac{y-x}{r}.
\end{equation*}
If $M$ is a locally $\mathscr H^k$-rectifiable set then we define the \emph{approximate tangent space to $M$ at $x$}, denoted by $\Tan(M,x)$, to be the set of limit points of the measures $r^{-k} \mu_{x,r}$ as $r \downarrow 0$ in the weak-$*$ topology. 
It is possible to prove (see, e.g. \cite[Theorem 10.2]{Maggi}) that for $\H^k$-a.e. $x \in M$ there exists a unique $k$-plane $\pi_x$ such that $\Tan(M,x) =\{ \H^k \rest_{\pi_x} \}$. 
We further emphasize that the approximate tangent space to a smooth set is related to the ordinary tangent space, in the sense of differential geometry. More precisely, we have the following
\begin{proposition}[{\cite[Proposition 2.88]{AFP}}]\label{prop:tang_smooth}
	Let $\phi \colon \R^k \to \R^d$ be a one-to-one Lipschitz function and let $D \subset \R^k$ be a $\L^k$-measurable set. Then $E = \Phi(D)$ satisfies 
	\begin{equation*}
	\Tan(E,x) = \{ \H^k \rest_{ d\phi_{\phi^{-1}(x)}(\R^k)}  \} \quad \text{ for } \H^{k} \text{-a.e. } x \in E, 
	\end{equation*}
	where $d\phi$ is the usual differential of $\phi$.
\end{proposition}

\subsection{Fine properties of sets of finite perimeter}
Given a Lebesgue measurable set $E \subseteq \R^d$ we define the \emph{upper/lower densities at $x$} by
\begin{equation*}
\overline{D}(E,x) := \limsup_{r\to 0} \frac{\L^d(E \cap B_r(x))}{\L^d(B_r(x))},\quad \quad \underline{D}(E,x) := \liminf_{r\to 0} \frac{\L^d(E \cap B_r(x))}{\L^d(B_r(x))},
\end{equation*}
and $D(E,x)$ denotes the common value of $\underline{D}(E,x)$ and $\overline{D}(E,x)$ whenever they are equal. In particular, we will denote by $E^t$, for $t \in [0,1]$, the set of points of density $t$
\begin{equation*}
E^{t} := \left\{x \in \R^d \;|\; D(E,x) =t \right\}.
\end{equation*} 

The \emph{essential exterior of $E$} is $E^{0}$ and the \emph{essential interior of $E$} is $E^{1}$. Ultimately, the \emph{essential boundary of $E$} is 
\begin{equation*}
\partial^e E := \R^d \setminus (E^{0} \cup E^{1}).
\end{equation*}

Following \cite[Definition 3.54]{AFP} we define the \emph{reduced boundary} of a set $E \subset \R^d$ to be the set of points $x \in \supp |D\1_E|$ such that the limit 
\begin{equation*}
\nu_E(x):= \lim_{r \downarrow 0} \frac{ D\1_E(B_r(x)) }{|D\1_E|(B_r(x)) }
\end{equation*}
exists in $\R^d$ and satisfies $|\nu_E(x)|=1$. We will denote by $\mathcal FE$ the reduced boundary and the function $\nu_E\colon \mathcal FE \to \mathbb S^{d-1}$ is called \emph{generalized inner normal} to $E$. 

The celebrated De Giorgi's Theorem can thus be stated as follows: 

\begin{theorem}[De Giorgi, {\cite[Theorem 3.59]{AFP}}] \label{thm:de_giorgi} Let $E$ be a Lebesgue measurable subset of $\R^d$ of finite perimeter in $\R^d$. Then $\mathcal F E$ is countably $(d-1)$-rectifiable and $\vert D\1_E \vert = \H^{d-1}\rest_{\mathcal F E}$. In addition, the approximate tangent space to $E$ at $x$ coincide with the orthogonal hyperplane to $\nu_E(x)$ for $\H^{d-1}$-a.e. $x \in \mathcal F E$, i.e. 
	\begin{equation*}
	\Tan(\mathcal F E,x) = \nu_E^\perp(x).   
	\end{equation*}
	%%%%SEE EQ 3.60 AFP 
\end{theorem}

The link between the reduced boundary, the essential boundary and the set of points of density $1/2$ is a remarkable theorem, due to Federer (see \cite[Theorem 3.61]{AFP}):
\begin{theorem}[Federer]\label{thm:de_giorgi_federer}
	%A set $E\subseteq \R^d$ has finite perimeter if and only $\H^{d-1}(\partial^e E) < \infty$. Furthermore, 
	If $E \subset \R^d$ has finite perimeter then 
	\begin{equation*}
	\mathcal F E \subset E^{1/2} \subset \partial^e E
	\end{equation*}
	and $\H^{d-1}(\partial^e E \setminus E^{1/2})=0$.
\end{theorem}

In particular, if $E\subset \R^d$ has finite perimeter then $\H^{d-1}(E^{1/2}) = \H^{d-1}(\mathcal F E) < \infty$. However it is known (see e.g. \cite[Theorem 6, (2)]{koliada}) that the condition $\H^{d-1}(E^{1/2})<\infty$ is not sufficient for $E\subset \R^d$ to have finite perimeter.

\begin{remark}
	Taking into account De Giorgi's Theorem \ref{thm:de_giorgi} we can write Coarea Formula for a function $u \in \BV(\R^d)$ in the following way (see e.g. \cite[Formula (3.63)]{AFP}: 
	\begin{equation}\label{eq:coarea_borel}
	|Du|(B) = \int_{\R} \H^{d-1}\left(\partial^e \{u \ge t\} \cap B\right) \, dt \qquad \text{ for every Borel set } B\subseteq \R^d.  
	\end{equation}
\end{remark}

\subsection{Indecomposable and simple sets}
From \cite{ACMM} we recall the following definitions. 
\begin{definition}[Decomposable and indecomposable sets]\label{def:indecomposable}
	A measurable set $E\subseteq \R^d$ of finite perimeter is called \emph{decomposable} if there exist two measurable sets $A,B\subseteq \R^d$ with strictly positive measure	such that $E = A \cup B$, $A \cap B = \emptyset$ and $P(E) = P(A) + P(B)$. A set $E$ which is not decomposable is called \emph{indecomposable}.
\end{definition}

Any set with finite perimeter can be decomposed into at most countably many indecomposable sets {\cite[Theorem 1]{ACMM}}: 

\begin{theorem}[Decomposition theorem]\label{thm:ACMM}
	Let $E$ be a set with finite perimeter in $\R^d$. Then there exists a unique (up to permutations) at most countable family of pairwise disjoint indecomposable sets $\{E_i\}_{i\in I}$ such that $\L^d(E_i)>0$, $E=\bigcup_{i\in I} E_i$ and $P(E) = \sum_{i} P(E_i)$. 
	% Moreover the sets $E_i$'s are maximal indecomposable sets, i.e. any indecomposable $F\subseteq E$ is contained in some $E_i$ up to a $\L^d$-negligible set.
	Moreover, for any indecomposable $F \subseteq E$ with $\L^d(F)>0$ there exists a unique $j\in I$ such that $\L^d(F\setminus E_j)=0$.
\end{theorem}

\begin{definition}
	The sets $E_i$ defined above are called the $M$-connected components of $E$. The set $\{E_i\}_{i\in I}$ is denoted by $\mathcal{CC}^M(E)$, without loss of generality $I\subseteq\{0,1,2,\ldots\}$ and $0\in I$.
\end{definition}

By Theorem \ref{thm:ACMM} the $M$-connected components of $E$ are maximal in the following sense: any indecomposable $F \subseteq E$ with $\L^d(F)>0$ is contained in exactly one of the $M$-connected components of $E$, up to Lebesgue negligible subsets. We refer the reader to \cite{ACMM} for a comparison between indecomposability and the topological notion of connectedness.

The statement of Decomposition theorem can be slightly strengthened with the following simple result from \cite[Proposition~3]{ACMM} (see also Equation (10) in \cite[Remark~1]{ACMM}):

\begin{proposition}\label{p-decomposition-th+}
Let $E \subseteq \R^d$ be a set with finite perimeter. Let $\mathcal{CC}^M(E) = \{E_i\}_{i\in I}$, where $I$ is at most countable. Then $P(\bigcup_{i\in I_1 \cup I_2} E_i) = P(\bigcup_{i\in I_1} E_i) + P(\bigcup_{i\in I_2} E_i)$for any disjoint sets $I_1, I_2 \subseteq I$.
\end{proposition}

\begin{definition}[Holes, saturation, simple sets]\label{def:simple}
	Let $E$ be an indecomposable set. Any $M$-connected component of $\R^d \setminus E$ with finite measure is called a \emph{hole} of $E$. The saturation $\mathop{\mathrm{sat}} (E)$ of $E$ is defined as union of $E$ and all its holes. The set $E$ is called \emph{saturated} if $\mathop{\mathrm{sat}} (E) = E$. Any indecomposable saturated subset of $\R^d$ is called \emph{simple}.
\end{definition}

Observe that simple sets are necessarily of finite perimeter; for $d>1$, the only simple set $E$ with $\L^d(E) = \infty$ is $E= \R^d$.

\subsection{Further facts on indecomposable and simple sets}

We finally collect in this paragraph some useful, different characterization of indecomposable and of simple sets. 
We begin by considering indecomposable sets and we present a lemma which will be useful later.

\begin{lemma}\label{lemma:equivalence_hausdorff}
	Let $F \subseteq E \subset \R^d$ be two sets of finite perimeter. Then 
	\begin{equation}\label{eq:assumptions}
	\partial^e F \subseteq \partial^e E \mod \H^{d-1} \iff \H^{d-1}(\partial^e F \cap E^1) = 0.  
	\end{equation}
	Furthermore, if $E$ is indecomposable and one (hence both) of \eqref{eq:assumptions} holds true, then $\L^d(F)=0$ or $\L^d(E \setminus F)=0$. 
\end{lemma}

\begin{proof}
	Let us show first the equivalence. First notice that from $F \subset E$, together with the monotonicity of the Lebesgue measure, we deduce $E^0 \subset F^0$. Hence, the following equalities hold modulo $\H^{d-1}$: 
	\begin{equation}\label{eq:inclusioni_hausd}
	\partial^e F = (\partial^e F \cap E^1) \cup (\partial^e F \cap E^0) \cup (\partial^e F \cap \partial^eE) = (\partial^e F \cap E^1) \cup (\partial^e F \cap \partial^eE) .
	\end{equation}
	From \eqref{eq:inclusioni_hausd} we easily get the equivalence: on the one hand, if $\partial^e F \subseteq \partial^e E$  then we must have 
	\begin{equation*}
	 (\partial^e F \cap E^1) \cup (\partial^e F \cap \partial^eE) = \partial^e F \subset \partial^e E \qquad \mod \H^{d-1}
	\end{equation*}
	and therefore the only possibility is that $\H^{d-1}(\partial^e F \cap E^1) = 0$. 
	Viceversa, if $\H^{d-1}(\partial^e F \cap E^1) = 0$, from \eqref{eq:inclusioni_hausd} we get 
	\begin{equation*}
	\partial^e F = (\partial^e F \cap E^1) \cup (\partial^e F \cap \partial^eE) = (\partial^e F \cap \partial^eE) \subset \partial^e E \qquad \mod \H^{d-1}
	\end{equation*}
	which is what we wanted. 
	Let us now turn to prove that there are no non-trivial subsets $F \subset E$ satisfying the conditions \eqref{eq:assumptions} if $E$ is indecomposable. Let $F \subseteq E$ be a set of finite perimeter with $\partial^e F \subseteq \partial^e E \mod \H^{d-1}$. Then it is easy to check that 
	\begin{equation}\label{eq:inclusion_diff}
	\partial^e(E \setminus F) \subset \partial^e E \setminus \partial^e F \mod \H^{d-1}.
	\end{equation}
	Let us show \eqref{eq:inclusion_diff}: on the one hand, it is clear	that $\partial^e(E \setminus F) \subset \partial^e E$. On the other hand, we show that 
	\begin{equation*}
	\H^{d-1}(\partial^e(E \setminus F) \cap \partial^e F) = 0.
	\end{equation*}
	Indeed, $\H^{d-1}$-a.e. $x \in \partial^e(E \setminus F)$ satisfies $D(E\setminus F,x) = \frac{1}{2}$ by De Giorgi-Federer's Theorem \ref{thm:de_giorgi_federer}. Similarly $\H^{d-1}$-a.e. $x \in \partial^e F$ satisfies $D(F,x) = \frac{1}{2}$: thus, for $\H^{d-1}$-a.e. $x \in (\partial^e(E \setminus F) \cap \partial^e F)$ we have 
	\begin{equation*}
	D(E,x) = D(F,x) + D(E\setminus F,x) = \frac{1}{2}+\frac{1}{2}=1
	\end{equation*} 
	which contradicts the fact that $\partial^e(E \setminus F) \cap \partial^e F \subseteq \partial^e E$. Having shown \eqref{eq:inclusion_diff} we get, taking Hausdorff measure of both sides, 
	\begin{equation*}
	\H^{d-1}( \partial^e(E \setminus F)) \le \H^{d-1}( \partial^e E \setminus \partial^e F) = \H^{d-1}(\partial^e E) - \H^{d-1}(\partial^e F)
	\end{equation*}
	or equivalently 
	\begin{equation*}
	P(E\setminus F) + P(F) \le P(E).
	\end{equation*}
	The other inequality is trivial by subadditivity of the perimeter, hence  $P(E\setminus F) + P(F) = P(E)$ which implies the desired conclusion, being $E$ indecomposable.
\end{proof}

\begin{proposition}[Dolzmann-M\"uller]\label{prop:dolzmann}
	A set $E \subset \R^d$ of finite perimeter is indecomposable if and only if for any $u \in \BV_\loc(\R^d)$ with $V(u) < \infty$ the following implication holds true: 
	\begin{equation*}
	|Du|(E^{1}) = 0 \quad \Longrightarrow \quad \exists c \in \R: \, u(x) = c \text{ for a.e. } x \in E. 
	\end{equation*}
\end{proposition}

\begin{proof}
Let $E$ be indecomposable and let $u \in \BV_{\loc}(\R^d)$ be a function with $|Du|(E^1) = 0$. Set $v:= u \1_E \in \BV_{\loc}(\R^d)$ and observe that, by Coarea Formula \eqref{eq:coarea_borel}, we have 
\begin{equation*}
| Dv | (E_1) = \int_{\R} \H^{d-1}\left (\partial^e (\{u \ge t \}\cap E) \cap E^1\right) \, dt \le  \int_{\R} \H^{d-1}\Big ( (\partial^e \{u \ge t \} \cup \partial^e E) \cap E^1\Big) \, dt, 
\end{equation*}
where we have used the elementary inclusion $\partial^e (\{u \ge t \}\cap E)  \subset \partial^e \{u \ge t \} \cup \partial^e E$. Taking into account that $\partial^e E \cap E^1 = \emptyset$ we can continue the above chain of inequalities as follows:
\begin{equation*}
| Dv | (E_1)  \le  \int_{\R} \H^{d-1}\left ( \partial^e \{u \ge t \} \cup \partial^e E) \cap E^1\right) \, dt = \int_{\R} \H^{d-1}\left ( \partial^e \{u \ge t \} \cap E^1\right) \, dt = |Du|(E^1) = 0
\end{equation*}
by Coarea Formula applied on $u$. Thus we have $\H^{d-1}\left (\partial^e (\{u \ge t \} \cap E) \cap E^1\right) = 0$ for a.e. $t \in \R$. Now we apply Lemma \ref{lemma:equivalence_hausdorff} to $F:=\{u \ge t \} \cap E\subset E$: since $E$ is indecomposable, we deduce 
\begin{equation*}
\L^d\left(  \{u \ge t \} \cap E  \right) = 0 \qquad \text{ or } \qquad  \L^d\left(  \{u < t \} \cap E  \right) = 0 
\end{equation*}
for a.e. $t \in \R$, which is easily seen to be equivalent to $u$ being constant in $E$.
\end{proof}

Concerning simple sets, we want to prove that simplicity for a set $E$ with $|E|<\infty$ is equivalent to indecomposability both of $E$ and of $E^c$. We need the following preliminary 
\begin{lemma}\label{lemma:unique_infinity}
	Let $E \subset \R^d$, $d>1$, be a set of finite perimeter and assume $\L^d(E)=+\infty$. Let $\mathcal C\mathcal C^M (E) = \{E_i\}_{i \in I}$ be the family of its indecomposable components. Then there exists a unique $j \in I$ such that $\L^d(E_j) = +\infty$.
\end{lemma}

\begin{proof}
	The statement is a consequence of the convergence of the series $\sum_{i \in I} P(E_i)$ and of the isoperimetric inequality. Indeed, by contradiction, let us assume that for every $i \in I$ it holds $\L^d (E_i) < \infty$. 
	In particular, for every $i \in I$ it has to be $\L^d( \R^d \setminus E_i)= +\infty$ and hence, by the isoperimetric inequality we would get 
	\begin{equation*}
	\L^d(E) = \L^d \left ( \bigcup_{i \in I} E_i \right) \le \sum_{i \in I} \L^d( E_i)  \le C_d \sum_{i \in I} P(E_i) \le C_dP(E) < \infty 
	\end{equation*}
	which is absurd. Hence there must exist at least one element $j \in I$ such that $\L^d(E_j) =+\infty$. 
	Let us now prove the uniqueness of $j$: assume that there exists $j_1, j_2 \in I$ such that $\L^d(E_{j_1}) = \L^d(E_{j_2}) = +\infty$. Since 
	\begin{equation*}
	P(E_{j_1})+ P(E_{j_2}) \le \sum_{i \in I} P(E_i) = P(E) < +\infty
	\end{equation*}
	we have that $P(E_{j_1}) < \infty$ and $P(E_{j_2})< \infty$. Furthermore, by definition of indecomposable components the sets $E_{j_1}$ and $E_{j_2}$ are (essentially) disjoint, i.e. $E_{j_1} \cap E_{j_2} = \emptyset$ mod $\mathscr L^d$. In particular, we deduce that 
	\begin{equation*}
	E_{j_2} \subset E_{j_1}^c \Rightarrow +\infty = \L^d(E_{j_2}) \le \L^d(E_{j_1}^c) 
	\end{equation*}
	so 
	\begin{equation*}
	\L^d(E_{j_1})  = \L^d( E_{j_1}^c) = +\infty
	\end{equation*} 
	which is a contradiction with the fact that $P(E_{j_1}) < \infty$ (and the isoperimetric inequality). Thus $j \in I$ has to be unique and the proposition is proved.
\end{proof}

We are now ready to present the following characterization of simple sets: 

\begin{proposition}\label{prop:char_simple}
	Let $E \subset \R^d$, $d>1$, be a set with finite positive measure, $\L^d(E) \in (0,+\infty)$. The set $E$ is simple if and only if $E$ and $E^c$ are indecomposable. 
\end{proposition}

\begin{proof} Assume that $E$ is simple. Then it is clearly indecomposable; thus it is sufficient to show that $E^c$ is indecomposable. Since $\L^d(E) \in (0,+\infty)$, we have $|E^c| = +\infty$. Letting $\mathcal C\mathcal C^M (E^c) := \{U_i\}_{i \in I}$ be the indecomposable components of $E^c$, by Lemma \ref{lemma:unique_infinity} there exists one and only one $j \in I$ such that $\L^d(U_j) = +\infty$. So if $\sharp I > 1$ the other components $\{U_i\}_{i \ne j}$ of $E^c$ must have finite measure, i.e. they are holes of $E$. This contradicts the simplicity of the set $E$: hence  $\sharp I = 1$ and $E^c$ is thus indecomposable. 

To prove the converse, let us now assume that $\L^d(E) \in (0,+\infty)$ and both $E$ and $E^c$ are indecomposable and we want to prove that $E$ has no holes. By definition a hole of $E$ is a indecomposable component of $E^c$ of finite measure. Being indecomposable, $E^c$ has a unique indecomposable component, which coincides with itself. But $\L^d(E^c) = \infty$ since $E$ has finite measure, and this implies that $E$ has no holes, hence it is simple. 
\end{proof}

\begin{remark}\label{rmk:still_holds} The necessary condition in Proposition \ref{prop:char_simple} holds even if $\L^d(E)=+\infty$ if $d>1$: indeed, as already observed, if $E$ is simple and $\L^d(E)=+\infty$ then $E=\R^d$, hence the claim is trivial, being the empty set indecomposable.\end{remark}
	
\subsection{Jordan curves in $\R^2$}
In this section we collect some results about Jordan curves in the plane $\R^2$.

\begin{definition}
	A set $C \subseteq \R^2$ is called a \emph{Jordan curve} if $C = \gamma([a,b])$ for some $a,b\in \R$ (with $a<b$) and some continuous map $\gamma\colon [a,b]\to \R^2$, one-to-one on $[a,b)$ and such that $\gamma(a) = \gamma(b)$.
\end{definition}

\begin{remark} If $\H^1(C) < \infty$ then $\gamma$ can be chosen in such a way that it is Lipschitz (see \cite[Lemma 3]{ACMM}), and in this case $\Gamma$ is called a \emph{rectifiable Jordan curve}.\end{remark}

Without any loss of generality, when dealing with Jordan curves, we will always suppose $[a,b]=[0,1]$. The following result, borrowed from \cite{ACMM}, will play a crucial role in the paper. 

\begin{theorem}[{\cite[Theorem 7]{ACMM}}]\label{thm:simple-boundary}
	Let $E \subseteq \R^2$ be a simple set with $\L^2(E) \in (0,+\infty)$.
	Then $E$ is essentially bounded and $\d^eE$ is equivalent, up to an $\H^1$-negligible set, to a rectifiable Jordan curve. Conversely, $\mathop{\mathsf{int}}(C)$ is a simple set for any rectifiable Jordan curve $C$.
\end{theorem}

Here $\mathop{\mathsf{int}}(C)$ denotes the bounded connected component of $\R^2 \setminus C$, given by the celebrated Jordan theorem (see e.g. \cite[Proposition 2B.1]{hatcher}). 

\subsection{Extreme points and Choquet theory}
In this section we recall the main facts about extreme points of compact, convex sets in normed spaces. Standard references are \cite{rudin, phelps}. 

Let $X$ be a topological vector space and let $K \subset X$. A point $x \in K$ is an \emph{extreme point of $K$} if  
	\begin{equation*}
	y,z \in K: \quad  t \in [0,1] \quad  x = (1-t)y + tz  \quad \Longrightarrow \quad x=y=z. 
	\end{equation*}
 The set of extreme points of $K$ will be denoted by $\ext K$.

 \begin{remark}[The set of extreme points is a Borel set]\label{rmk:extreme_borel}
 	Recall that $\ext K$ is a Borel subset of $K$ if the topology of $X$ is induced by some metric $\rho$. Indeed, the set $K\setminus \ext K$ can be written as $\bigcup_n C_n$, where
 	\begin{equation*}
 	C_n := \bigg\{\frac{y+z}{2}\;\bigg|\;y,z\in K,\; \rho (y,z) \geq \frac{1}{n} \bigg\} \quad\text{ for every }n\in\N \text{ with } n \ge 1.
 	\end{equation*}
 	Given that each set $C_n$ is a closed subset of $X$ we conclude that $\ext K$ is Borel.
 \end{remark}

In the case $K$ is a convex, compact set the (closed convex hull of the) set of extreme points of $K$ coincides with the set $K$ itself, as the following theorem states: 

\begin{theorem}[Krein-Milman]\label{thm:krein_milman} If $K \subset X$ is non-empty, compact, convex set then $K = \overline{\mathsf {co}}(\ext(K))$.
\end{theorem}

We recall that, in a vector space $X$, the convex hull $\mathsf{co}(A)$ of a set $A \subset X$ is the intersection of all convex sets containing $A$. 

\begin{definition}[Vector valued integration]\label{def:vector_integration}
	Let $\mu$ be a measure on a non-empty set $Q$. Let $f \colon Q \to X$ be an $X$-valued function such that $(\Lambda f)(q) := \Lambda(f(q))$ is $\mu$-integrable for every $\Lambda \colon X \to \R$ linear and continuous. If there exists $y \in X$ such that 
	\begin{equation*}
	\Lambda y = \int_Q \Lambda f \, d\mu 
	\end{equation*}
	for every $\Lambda \colon X \to \R$ linear and continuous then we say that $y$ is the \emph{integral of $f$ with respect to $\mu$} and we write   
	\begin{equation*}
	\int_Q f d\mu := y. 
	\end{equation*}
\end{definition}

\begin{theorem}[Representation of the convex hull]\label{thm:convex_hull_representation} Let $Q \subset X$ be a compact set and let $H:=\mathsf {co}(Q)$. Assume that $\overline H =  \overline{\mathsf {co}}(Q)$ is compact as well. Then 
	\begin{equation*}
	y \in \overline{H} \iff \exists \mu \in \mathscr P(Q): \, y = \int_Q x \, d\mu(x).
	\end{equation*}
\end{theorem}

One of the fundamental results in functional analysis and convex analysis is the
following theorem, which can be obtained combining Theorem \ref{thm:krein_milman} with Theorem \ref{thm:convex_hull_representation}: 

\begin{theorem}[Choquet \cite{phelps}]\label{thm:Choquet}
	Let $X$ be a metrizable topological vector space and let $\emptyset\neq K\subset X$ be convex and compact.
	Then for any point $x\in K$ there exists a Borel probability measure
	$\mu$ on $X$ (possibly depending on $x$), which is concentrated on $\ext K$ and satisfies
	\begin{equation*}
	x = \int_{\ext K} y \, d\mu(y)
	\end{equation*}
	where the integral is understood in the sense of Definition \ref{def:vector_integration}, i.e. explicitly
	\begin{equation*}
	\Lambda(x) = \int_{\ext K} \Lambda(y) \, d\mu(y), \qquad \text{ for every } \Lambda \colon X \to \R \text{ linear and continuous}.
	\end{equation*}
\end{theorem}

\begin{remark}[Extreme points and isomorphisms]\label{rmk:extreme_isomor}
	Let $(Y, \Vert \cdot \Vert_{Y})$ be a normed space and suppose that $\phi \colon X \to Y$ is a linear isomorphism between $X$ and $Y$. Then for any set $A \subset X$ it holds 
	\begin{equation*}
	\ext \phi(A) = \phi(\ext A).
	\end{equation*}
	Indeed, consider $z \in \ext \phi(A)$. Being $\phi$ one-to-one and onto, there exists a unique $a \in A$ such that $z = \phi(a)$. We want to prove that $a \in \ext A$: let 
	\begin{equation*}
	f,g \in A: \, \lambda f + (1-\lambda)g = a.  
	\end{equation*} 
	Since $\phi$ is linear 
	\begin{equation*}
	\lambda \phi(f) + (1-\lambda)\phi(g) = \phi(a)=z
	\end{equation*}
	but $z$ was an extreme point hence $\phi(f)=\phi(g)=z$ which implies $f=g=a$, i.e. $a$ is also extreme. An analogous proof shows that if $b \in \ext A$ than $\phi(b)$ is also extreme of $\phi(A)$.
\end{remark}

\section{Extreme points of the unit ball of \texorpdfstring{$\BV$}{BV} functions in \texorpdfstring{$\R^d$}{Rd}}

Let us consider the Banach space $X := \left(\BV(\R^d), \Vert \cdot \Vert_{\BV}\right)$ and let us characterize extreme points of $B_1^X$, the closed unit ball.

\begin{proposition}[Extreme points of unit ball in $\BV(\R^d)$]\label{prop:extreme_points_of_BV} A function $f \in X$ is an extreme point of $B_1^X$ if and only there exists an indecomposable set $E \subset \R^d$ of positive, finite perimeter and positive, finite measure and a constant $\sigma \in \{\pm 1\}$ such that 
\begin{equation*}
	f(x) = \sigma \frac{\1_E(x)}{\| \1_E \|_{\BV}}, \qquad \L^d\text{-a.e. } x \in \R^d. 
\end{equation*}
\end{proposition}

We will need the following auxiliary 

\begin{lemma}\label{lemma:chin_rule}
Let $f \in \BV_{\loc}(\R^d)$ and let, for any $\lambda \in \R$, 
\begin{equation*}
f_{\lambda}^+ :=\max\{f-\lambda,0\} \qquad \text { and } f_{\lambda}^- := f-f_\lambda^+ =  \min\{\lambda, f\}.
\end{equation*}
Then for every open set $\Omega \Subset \R^d$ it holds  
\begin{equation}\label{eq:chain_modificata}
|Df|(\Omega) = |D(f_{\lambda}^+)|(\Omega) + |D(f_{\lambda}^-)|(\Omega). 
\end{equation}
\end{lemma}

\begin{proof}
	To begin we consider the case $\lambda=0$ and we notice that, in this case, the decomposition of $f$ into $f^+_\lambda + f_\lambda^-$ coincides with the standard decomposition into positive/negative part: 
	\begin{equation*}
	f_{0}^+ = f^+ \quad \text{ and } \quad  f_0^- := - f^-. 
	\end{equation*}	
	If $f \in W_{\loc}^{1,1}(\R^d)$ then, fixed $\Omega \Subset \R^d$, it is enough to apply the Chain Rule Theorem \cite[Section 4.2.2, Theorem 4(iii)]{EG}. 
	For the general case, consider a sequence $(f_n)_n \subset W^{1,1}(\Omega) \cap C^\infty(\Omega)$ with $f_n \to f$ strongly in $L^1(\Omega)$ and $|Df_n|(\Omega) \to |Df|(\Omega)$ (such a sequence can be obtained using Anzellotti-Giaquinta's Theorem, see e.g. \cite[Section 5.2.2]{EG}). Then for every $n \in \N$ it holds 
	\begin{equation*}
	| Df_n |(\Omega) = | D(f_n^+) |(\Omega) + |D(f_n^-)|(\Omega) 
	\end{equation*}  
	hence 
	\begin{equation*}
	\begin{split}
	| Df|(\Omega) = \lim_n | Df_n |(\Omega) & = \liminf_n | Df_n |(\Omega) \\
	& = \liminf_n ( | D(f_n^+) |(\Omega)  + |D(f_n^-)|(\Omega) )\\
	& \ge \liminf_n | D(f_n^+) |(\Omega) +\liminf_n  |D(f_n^-)|(\Omega) \\
	& \ge | D(f^+)|(\Omega) + | D(f^-)|(\Omega) 
	\end{split}
	\end{equation*}
	where the last inequality is a consequence of the l.s.c. of the total variation, since $f_n^+ \to f^+$ and $f_n^- \to f^-$ in $L^1(\Omega)$.
	The statement is thus proved for $\lambda=0$; to obtain the general case, we can apply the above claim to the function $g:=f-\lambda \in \BV_\loc(\R^d)$, noticing that 
	\begin{equation*}
	g^+ = f^+_\lambda, \qquad g^- = \lambda -  f^-_\lambda
	\end{equation*}
	and 
	\begin{equation*}
	Dg = Df, \qquad Dg^+ = Df^+_\lambda, \qquad  Dg^- = -Df^-_\lambda. 
	\end{equation*}
	whence \eqref{eq:chain_modificata}.
\end{proof}

We now show the following Lemma, which ensures that extreme points lie in the set of normalized indicators of sets of finite perimeter. Recall that for any set of finite perimeter $E \subset \R^d$, either $E$ or $E^c$ has finite Lebesgue measure by the isoperimetric inequality \eqref{eq:isoperimetry}. 

\begin{lemma}\label{lemma:extreme_in_perimeters}
	Let $f\in X$ be an extreme point of the closed unit ball $B_1^X$. Then there exists a set $E\subseteq \R^d$ with positive, finite perimeter and positive, finite measure $\L^d(E)<\infty$ and a constant $\sigma \in \{\pm 1\}$ such that $f = \sigma \frac{1}{\| \1_E\|_{\BV}}\1_E$. 
\end{lemma}

\begin{proof}
	\emph{Step 1. Any extreme function has constant sign.}
	Let $f\in X$ be extreme of $B_1^X$. Then, by standard facts, we have necessarily $\|f\|_{\BV}=1$.
	Let us decompose $f$ into positive and negative part as $f=f^+ - f^-$. By the very definition of Lebesgue integral for signed functions we have that 
	\begin{equation*}
	\| f \|_{1} = \| f^+ \|_1 + \|f^-\|_1 
	\end{equation*}
	while, by Lemma \ref{lemma:chin_rule} with $\lambda=0$, we have that 
	\begin{equation*}
	\|Df\|_{\M} = \|Df^+\|_{\M} + \|Df^-\|_{\M}.  
	\end{equation*}
	Adding up the two equalities, we find out that
	\begin{equation*}
	\| f \|_{\BV} = \| f^+ \|_{\BV} + \|f^-\|_{\BV} 
	\end{equation*}
	and this can be used to decompose $f$ into a convex linear combination of two \emph{signed} functions with unit $\BV$ norm: 
	\begin{equation}\label{eq:convex}
	f = \|f^+\|_{\BV} \cdot \frac{f^+}{\|f^+\|_{\BV}} + \|f^-\|_{\BV} \cdot \frac{-f^-}{\| f^-\|_{\BV}}.
	\end{equation}
	Hence any extremal point is necessarily a function with constant sign and, without any loss of generality, we consider $f \ge 0$.
	
	\emph{Step 2. Any extreme function attains at most one non-zero value.}
	 We now would like to prove that $f(x) \in \{0, \alpha\}$ for some $\alpha>0$ for $\mathscr L^d$- a.e. $x \in \R^d$.
		
	Suppose by contradiction that it is not true: hence, there exist two points $x_1, x_2$ such that $f(x_1) \ne 0$, $f(x_2) \ne 0$ and also $f(x_1) \ne f(x_2)$. Without any loss of generality, suppose $f(x_1) < f(x_2)$. We can also assume that $x_1, x_2$ are Lebesgue points of $f$ (this property being satisfied almost everywhere by standard facts).
	Consider an arbitrary $\lambda \in (f(x_1), f(x_2))$ and define the non-negative functions 
	\begin{equation*}
	f_{\lambda}^+ :=\max\{f-\lambda,0\} \qquad \text { and } f_{\lambda}^- := f-f_\lambda^+ =  \min\{\lambda, f\}. 
	\end{equation*}
	By Lemma \ref{lemma:chin_rule} we deduce
	\begin{equation*}
	\|Df\|_{\M} = \|D(f_{\lambda}^+)\|_{\M} + \|D(f_\lambda^-)\|_{\M},
	\end{equation*}
	while from the pointwise equality $f_\lambda^+ + f_{\lambda}^- = f$, together with non-negativity, we get 
	\begin{equation*}
	\| f \|_{1} = \| f^+_\lambda \|_1 + \|f^-_\lambda\|_1,
	\end{equation*}
	and thus 
	\begin{equation*}
	\| f \|_{\BV} = \| f^+_\lambda \|_{\BV}+ \|f^-_\lambda\|_{\BV}. 
	\end{equation*}
	In particular, we can decompose 
	\begin{equation}\label{eq:convex_lambda}
	f = \|f^+_\lambda\|_{\BV} \cdot \underbrace{ \frac{f^+_\lambda}{\|f^+_\lambda\|_{\BV}}}_{\in B_1^X} + \|f^-_\lambda\|_{\BV} \cdot \underbrace{\frac{f^-_\lambda}{\| f^-_\lambda\|_{\BV}}}_{\in B_1^X}.
	\end{equation}
	Notice that the choice of $\lambda \in (f(x_1), f(x_2))$ together with the fact that $f(x_1)\ne 0 \ne f(x_2)$ grant that the decomposition \eqref{eq:convex_lambda} is non-trivial and well-posed, in the sense that: 
	
	\begin{enumerate}
	\item the functions $f^\pm_\lambda$ are linearly independent: if $a f^+_\lambda + b f^-_\lambda = 0$ for $a,b \in \R$, then evaluating at $x_1$ we deduce 
	\begin{equation*}
	b f(x_1) = 0 \Rightarrow b = 0
	\end{equation*}
	and evaluation at $x_2$ yields 
	\begin{equation*}
	a (f(x_2)-\lambda) = 0\Rightarrow a = 0; 
	\end{equation*}
	
	\item we have  $\Vert f^\pm_\lambda \Vert_{\BV} >0$: indeed, if it were e.g.  $\|f^-_\lambda\|_{\BV}=0$, then $f = f^+_\lambda$ a.e. which means 
	\begin{equation}\label{eq:geq}
	f(x) \ge \lambda \qquad \text{ for } \mathscr L^d\text{-a.e.} \,  x \in \R^d. 
	\end{equation}
	On the other hand, $x_1$ is a Lebesgue point of $f$ with Lebesgue value $f(x_1) < \lambda$, so by definition 
	\begin{equation*}
	\lambda > f(x_1) = \lim_{r \to 0} \fint_{B_r(x_1)} f(y) \, dy \overset{\eqref{eq:geq}}{\ge} \lim_{r \to 0} \fint_{B_r(x_1)} \lambda \, dy =\lambda, 
	\end{equation*}
	which is a contradiction. 
	\end{enumerate}

		Thus \eqref{eq:convex_lambda} is a non-trivial, convex decomposition of $f$ which contradicts extremality: the contradiction stems from the assumption that there exists two points $x_1,x_2$ such that $f(x_1) \ne 0$, $f(x_2) \ne 0$ and $f(x_1)\ne f(x_2)$. So we must have $f(x) \in \{0,\alpha\}$ for a.e. $x$ for some $\alpha >0$. 
		
		\emph{Step 3. Any extreme function is an indicator function.} From \emph{Step 2} we immediately deduce 
		\begin{equation*}
		f(x) = \alpha \1_{E}, \qquad \text{ where } E:=\{x \in \R^d: f(x)=\alpha\}.
		\end{equation*}
		The set $E$ has finite perimeter because $f \in \BV(\R^d)$	and, being $\|f\|_{\BV} = 1$, we deduce that necessarily $\alpha= \|\1_E\|_{\BV}^{-1}$. This concludes the proof. 
\end{proof}

We can now prove the main result of this section. 

\begin{proof}[Proof of Proposition \ref{prop:extreme_points_of_BV}] We split the proof into two steps. 
	
	\emph{Sufficiency.} Let $E \subset \R^d$ be a set of positive, finite perimeter and assume it is indecomposable. Let $c=\frac{1}{P(E)}$ and let us prove that $f := c \1_E$ is an extreme point of $B_1^X$. Assume that for some functions $g,h \in B_1^X$ and $\lambda \in [0,1]$ we can write 
	\begin{equation*}
	f = \lambda g + (1-\lambda)h
	\end{equation*}
	and let us prove that necessarily $g=c\1_E$ and $h=c\1_E$. 
	Since $\| f \|_{\BV} = 1$ we have that 
	\begin{equation*}
	1 \le \lambda \| g \|_{\BV} + (1-\lambda) \|h\|_{\BV}
	\end{equation*}
	and we claim that actually equality holds. If it were 
	\begin{equation*}
	1 < \lambda \| g \|_{\BV} + (1-\lambda) \|h\|_{\BV} 
	\end{equation*}
	then we would get, being $f,g \in B_1^X$
	\begin{equation*}
	1 < \lambda \| g \|_{\BV} + (1-\lambda) \|h\|_{\BV} \le \lambda + (1-\lambda) = 1,
	\end{equation*}
	a contradiction. In a complete similar way, one can prove that $\|g\|_{\BV}=1=\|h\|_{\BV}$. All in all, we can represent 
	\begin{equation}\label{eq:sum_functions_BV}
	\1_E = \phi + \psi 
	\end{equation}
	with 
	\begin{equation}\label{eq:additive_BV}
	\| \1_E \|_{\BV}= \|\phi\|_{\BV} + \|\psi\|_{\BV} 
	\end{equation}
	being $\phi = c^{-1}\lambda g$ and $\psi = c^{-1}(1-\lambda) h$. Notice that $\phi, \psi$ have the same sign a.e., otherwise we would have 
	\begin{equation*}
	\int_{\R^d} \vert\phi(x) + \psi(x)\vert \, dx <	\int_{\R^d} | \phi(x) | + |\psi(x)| \, dx  
	\end{equation*}
	which would yield
	\begin{equation*}
	 \|\phi + \psi\|_{\BV} =  \|\phi + \psi\|_{1} + \| D (\phi + \psi)\|_{\M} <   \|\phi \|_1 + \| \psi\|_{1} + \| D\phi \|_{\M} + \| D\psi\|_{\M} =  \|\phi \|_{\BV} +\| \psi\|_{\BV}, 
	\end{equation*}
	contradicting \eqref{eq:additive_BV}. Since $\1_E = \phi + \psi$, we have $\phi, \psi \ge 0$ a.e. and therefore
	\begin{equation}\label{eq:zero_out}
	\phi = \psi = 0 \qquad \text{ a.e. on } E^c. 
	\end{equation} 
	Notice furthermore that it holds 
\begin{equation}\label{eq:sum_of_measures_BV}
| D \1_E | = |D\phi| + |D\psi| \qquad \text{ as measures in } \R^d. 
\end{equation}
Indeed, by \eqref{eq:sum_functions_BV} and the triangle inequality we get $| D \1_E | \le |D\phi + D\psi| \le |D\phi| + |D\psi|$; the converse inequality then follows exploiting \eqref{eq:additive_BV}. 
In particular, computing \eqref{eq:sum_of_measures_BV} on the Borel set $E^{1}$ it follows
\begin{equation*}
|D\phi|(E^{1}) + |D\psi|(E^{1})  = | D \1_E |(E^{1}) = 0
\end{equation*}
where the last equality follows from De Giorgi's Theorem. Hence
\begin{equation*}
|D\phi|(E^{1}) = 0 = |D\psi|(E^{1}).
\end{equation*}
By Proposition \ref{prop:dolzmann} and by the indecomposability of $E$, there exist constants $c_1,c_2 \in \R$ such that 
\begin{equation}\label{eq:sets_BV}
\phi(x) = c_1, \qquad \psi(x) = c_2 \qquad \text{ a.e. in } E. 
\end{equation}
In particular, combining \eqref{eq:sets_BV} together with \eqref{eq:zero_out} we obtain 
\begin{equation*}
\phi(x) = c_1 \1_E(x), \qquad \psi(x) = c_2\1_E(x) \qquad \text{ a.e. in } \R^d 
\end{equation*}
and this in turn implies that 
\begin{equation*}
g = \alpha \1_E(x),
\end{equation*}
for some $\alpha  \in \R$. Being $\Vert g \Vert_{\BV}=1$ we obtain that the constant has to be 
\begin{equation*}
\alpha =\frac{1}{P(E)}.
\end{equation*}
One can argue similarly with $h$ and the conclusion is now achieved: we have proved that the only convex combination of elements in $B_1^X$ representing $f$ is the trivial one, i.e. $f$ is an extreme point of $B_1^X$.
 
	\emph{Necessity.} By Lemma \ref{lemma:extreme_in_perimeters}, we can already infer that there exists a set $E\subseteq \R^d$ with finite perimeter and $\sigma \in \{\pm 1\}$ such that $f =\sigma \frac{1}{\| \1_E\|_{\BV}}\1_E$ a.e. w.r.t. the Lebesgue measure. Now we prove that $E$ is indecomposable. Suppose by contradiction that $E$ is a decomposable set, i.e. $E = A \cup B$ with $A\cap B = \emptyset$ and $P(E) = P(A) + P(B)$. Since by additivity of the Lebesgue measure it holds $\L^d(E) = \L^d(A) + \L^d(B)$, we have 
	\begin{equation*}
	\|\1_E \|_{\BV} = 	\|\1_A \|_{\BV} +	\|\1_B \|_{\BV}.
	\end{equation*}
	Hence
	\begin{equation*}
	\frac{1}{\|\1_E \|_{\BV} } \1_E = \frac{\|\1_A \|_{\BV} }{\|\1_E \|_{\BV} } \underbrace{\frac{\1_A}{\|\1_A \|_{\BV} }}_{\in B_1^X} + \frac{\|\1_B \|_{\BV} }{\|\1_E \|_{\BV} } \underbrace{\frac{\1_B}{\|\1_B \|_{\BV} }}_{\in B_1^X}
	\end{equation*}
	is a representation of $f$ as a non-trivial convex combination of elements of $B_1^X$, contradicting extremality. Therefore if $\frac{\1_E}{P(E)}$ is an extreme point then $E$ has to be indecomposable.
\end{proof}

\section{Extreme points of the unit ball of \texorpdfstring{$\FV$}{FV} functions in \texorpdfstring{$\R^d$}{Rd}}

\begin{definition} We define the space $\FV(\R^d)$ as the function space 
	\begin{equation*}
	\FV(\R^d) :=  \{f \in L^{1^*}(\R^d): V(f)<+\infty\}. 
 	\end{equation*}
 \end{definition}

We recall $V(f)=V(f, \R^d)$ is the variation of a locally integrable function, see Definition \ref{def:variation}, while $1*$ is defined in \eqref{eq:sobolev_conj}. 
 
\begin{remark}
	It is easy to see that $\BV(\R^d) \subset \FV(\R^d) \subset \BV_{\loc}(\R^d)$ and both inclusion are strict. Indeed, any constant function is certainly locally integrable with zero total variation, but it is not in $L^p(\R^d)$ for any $p$. On the other hand, the function $f \colon \R^d \to \R$ defined by 
	\begin{equation*}
	f(x)= g(|x|), \qquad \text{ where } g(s) := \min\left\{1, \frac{1}{s^d}\right\} 
	\end{equation*}
	is in $\FV(\R^d)$ but not in $\BV(\R^d)$. Let us verify this claim: 
	\begin{equation*}
	\int_{\R^d} f(x) dx = C_d \int_{0}^{+\infty} g(s) s^{d-1} \, ds = + \infty  
	\end{equation*}
	while, taking into account that $1^* := \nicefrac{d}{d-1}$ we have 
	\begin{equation*}
	\int_{\R^d} |f(x)|^{1^*} dx = C_d \int_{0}^{+\infty} g(s)^{\nicefrac{d}{d-1}} s^{d-1} \, ds = \widetilde{C_d} \left( 1+  \int_{1}^{+\infty}  s^{\nicefrac{1-2d}{d-1}} \, ds  \right) < +\infty. 
	\end{equation*}
	Notice that the variation of $f$ is finite
	\begin{equation*}
	V(f) = C_d \int_{1}^{+\infty} \frac{1}{s^{d+1}} s^{d-1} \, ds = C_d \int_{1}^{+\infty}  \frac{1}{s^2}<+ \infty . 
	\end{equation*}
 	\end{remark}

We now prove that the map $\Vert \cdot \Vert_{\FV} \colon \FV(\R^d)\ni f \mapsto \Vert f \Vert_{\FV}= V(f)$ gives to $\FV(\R^d)$ the structure of a normed space. 
\begin{proposition} 
The space $Y := \left(\FV(\R^d), \Vert \cdot \Vert_{\FV}\right)$ is a normed space. 
\end{proposition}

\begin{proof}
	Positivity and 1-homogeneity are clear from the definition of $\Vert \cdot \Vert_{\FV}$ and the triangle inequality as well. We have to prove only definiteness: for, let $f \in \FV(\R^d)$ with 
	\begin{equation}
	\Vert f \Vert_{\FV} = V(f) = 0.
	\end{equation}
	Applying Theorem \ref{thm:347AFP} we deduce that there exist $m \in \R$ and a constant $\gamma >0$ such that 
	\begin{equation*}
	\| f - m \|_{1^*} \le \gamma V(f)
	\end{equation*}
	whence $\| f - m \|_{1^*}=0$ and $f=m$ almost everywhere. Being $f \in L^{1^*}(\R^d)$ the only possibility is that $m=0$ hence the proposition is proved.
\end{proof}

We now aim at characterizing extreme points of $B_1^Y$, the closed unit ball in $Y$. Observe that, if $f$ is the characteristic function of a measurable set $A$, variation and perimeter coincide, i.e. 
\begin{equation*}
V(\1_A) = P(A,\R^d) =P(A). 
\end{equation*}

\begin{proposition}[Extreme points of unit ball in $\FV(\R^d)$]\label{prop:extreme_points_of_FV} A function $f \in Y$ is an extreme point of $B_1^Y$ if and only there exists a simple set $E \subset \R^d$ of positive, finite perimeter and a constant $\sigma \in \{\pm 1\}$ such that 
	\begin{equation*}
	f(x) = \sigma \frac{\1_E(x)}{P(E)}, \qquad \L^d\text{-a.e. } x \in \R^d. 
	\end{equation*}
\end{proposition}

\begin{proof} 
\emph{Sufficiency.} Let $E \subset \R^d$ be a simple set. Let $c=\frac{1}{P(E)}$ and let us prove that $f := c \1_E$ is an extreme point of $B_1^Y$. Assume that for some functions $g,h \in B_1^Y$ and $\lambda \in [0,1]$ we can write 
\begin{equation*}
f = \lambda g + (1-\lambda)h
\end{equation*}
and let us prove that necessarily $g=c\1_E$ and $h=c\1_E$. Since $\| f \|_{\FV} = 1$ we have that 
\begin{equation*}
1 \le \lambda \| g \|_{\FV} + (1-\lambda) \|h\|_{\FV}
\end{equation*}
and we claim that actually equality holds. If it were 
\begin{equation*}
1 < \lambda \| g \|_{\FV} + (1-\lambda) \|h\|_{\FV} 
\end{equation*}
then we would get, being $f,g \in B_1^Y$
\begin{equation*}
1 < \lambda \| g \|_{\FV} + (1-\lambda) \|h\|_{\FV} \le \lambda + (1-\lambda) = 1,
\end{equation*}
a contradiction. In a complete similar way, one can prove that $\|g\|_{\FV}=1=\|h\|_{\FV}$.
All in all, we can represent 
\begin{equation}\label{eq:sum_functions}
\1_E = \phi + \psi 
\end{equation}
with 
\begin{equation}\label{eq:additive_FV}
\| \1_E \|_{\FV}= \|\phi\|_{\FV} + \|\psi\|_{\FV} 
\end{equation}
being $\phi = c^{-1}\lambda g$ and $\psi = c^{-1}(1-\lambda) h$. Notice actually that it holds 
\begin{equation}\label{eq:sum_of_measures}
| D \1_E | = |D\phi| + |D\psi| \qquad \text{ as measures in } \R^d. 
\end{equation}
Indeed, by \eqref{eq:sum_functions} and the triangle inequality we get $| D \1_E | \le |D\phi + D\psi| \le |D\phi| + |D\psi|$; the converse inequality then follows exploiting \eqref{eq:additive_FV}. 
In particular, computing \eqref{eq:sum_of_measures} on the Borel set $E^{1}$ it follows
\begin{equation*}
|D\phi|(E^{1}) + |D\psi|(E^{1})  = | D \1_E |(E^{1}) = 0
\end{equation*}
where the last equality follows from De Giorgi's Theorem. Hence
\begin{equation*}
|D\phi|(E^{1}) = 0 = |D\psi|(E^{1}).
\end{equation*}
By Proposition \ref{prop:dolzmann} and by the indecomposability of $E$, there exist constants $c_1,c_2 \in \R$ such that 
\begin{equation}\label{eq:sets}
\phi(x) = c_1, \qquad \psi(x) = c_2 \qquad \text{ a.e. in } E. 
\end{equation}
In particular, $c_1+c_2=1$. In an analogous way, we also get $|D\phi|(E^{0}) = 0 = |D\psi|(E^{0})$: being $E^{0} = (\R^d \setminus E)^{1}$, by indecomposability of $E^c$ (recall Proposition \ref{prop:char_simple}), we conclude again by Proposition \ref{prop:dolzmann} that there exist constants $c_3,c_4 \in \R$ such that 
\begin{equation*}
\phi(x) = c_3, \qquad \psi(x) = c_4 \qquad \text{ a.e. in } E^c. 
\end{equation*}
By the Isoperimetric Inequality \eqref{eq:isoperimetry}, either $E$ or $E^c$ has finite measure and, up to rename everything, consider the case in which $E$ has finite measure. Then $E^c$ must have infinite Lebesgue measure and the functions $\phi,\psi$ are constant functions which are in $L^{1^*}(\R^d)$: thus it must be $c_3=c_4=0$, i.e. 
\begin{equation*}
\phi(x) = 0 = \psi(x) \qquad \text{ a.e. in } E^c. 
\end{equation*}
Combined with \eqref{eq:sets}, this gives that 
\begin{equation*}
\phi(x) = c_1 \1_E(x), \qquad \psi(x) = (1-c_1)\1_E(x) \qquad \text{ a.e. in } \R^d.
\end{equation*}
In particular, we deduce that 
\begin{equation*}
g = \alpha \1_E(x)
\end{equation*}
and being $\Vert g \Vert_{\FV}=1$ we obtain that the constant has to be 
\begin{equation*}
\alpha =\frac{1}{P(E)}.
\end{equation*}
One can argue similarly with $h$ and the conclusion is now achieved: we have proved that the only convex combination of elements in $B_1^Y$ representing $f$ is the trivial one, i.e. $f$ is an extreme point of $B_1^Y$. 
 
\emph{Necessity.} The argument used in the proof of \ref{lemma:extreme_in_perimeters} can be repeated verbatim here, yielding an analogous conclusion: an extreme point $f$ of $B_1^Y$ has necessarily the form 
\begin{equation*}
	f(x) = \sigma \frac{\1_E(x)}{P(E)}, \qquad \L^d\text{-a.e. } x \in \R^d
\end{equation*}
for some set of positive finite perimeter $E \subset \R^d$. It remains thus to show that $E$ has to be simple.  Let us show first that $E$ is indecomposable. Assume that it can be written as $E = A \cup B$ with $A\cap B = \emptyset$ and $P(E) = P(A) + P(B)$.
Hence $\frac{1}{P(E)} \1_E = \frac{P(A)}{P(E)} \frac{1}{P(A)} \1_A + \frac{P(B)}{P(E)} \frac{1}{P(B)} \1_B$ is a convex linear combination of indicators of sets (normalized by perimeter). Therefore if $\frac{1}{P(E)} \1_E$ is an extreme point of the unit ball in $Y$ then $E$ has to be indecomposable.

In view of Proposition \ref{prop:char_simple}, it remains to show that $E^c$ has to be indecomposable, too. For let us suppose that $C,D$ are such that $E^c = C \cup D$ with $C\cap D = \emptyset$ and $P(E^c) = P(C) + P(D)$. Arguing as above, we get that
\begin{equation*}
E = C^c \cap D^c = C' \setminus D \Rightarrow \1_E = \1_{C'} - \1_D, 
\end{equation*}
with $C' = C^c$. 
Consequently, since $P(E)=P(E^c) = P(C) + P(D)$ it holds
\begin{equation*}
\frac{1}{P(E)}\1_{E} = \frac{P(C)}{P(E)} \frac{1}{P(C)} \1_{C'} + \frac{P(D)}{P(E)} \frac{-1}{P(D)} \1_{D}
\end{equation*}
is a convex linear combination of indicators of sets (normalized by perimeter).
Therefore if $\frac{1}{P(E)} \1_E$ is an extremal point of then $E,E^c$ have to be indecomposable, hence $E$ is simple and this concludes the proof.
\end{proof} 
 
\section{Hamiltonian potential of divergence-free vector measures in \texorpdfstring{$\R^2$}{R2}} 

\subsection{Divergence-free measures and $\FV$}

We now define the space of vector-valued divergence-free measures. 
\begin{definition}\label{def:solenoids}We will denote by $\mathcal J(\R^d)$ the following set of vector valued measures: 
\begin{equation*}
\mathcal J(\R^d) := \{\bm \mu \in \mathscr M(\R^d;\R^d): \, \dive \bm \mu = 0 \}
\end{equation*}
where the divergence operator is understood in the sense of distributions. 
\end{definition}
The space $\mathcal J$ is a real vector space under the usual operations of additions of measures and multiplication by real numbers and it can be equipped with a norm given by the total variation: 
\begin{equation*}
\|\bm \mu\|_{\mathcal J}:= |\bm \mu|(\R^2). 
\end{equation*}

\begin{remark} As already observed in the Introduction, an important (somehow paradigmatic) example of a measure belonging to $\mathcal J$ is the one associated to a Lipschitz closed curve: if $\gamma \colon [0,1] \to \R^2$ is a Lipschitz map, injective on $[0,1)$ and with $\gamma(0)=\gamma(1)$ we can define the measure $\bm \mu_\gamma \in \mathcal M(\R^2; \R^2)$ to be 
\begin{equation*}
\langle \Phi , \bm \mu_\gamma \rangle := \int_{\R^2} \Phi(z) \, d\H^1 \rest_{\gamma([0,1])}(z) \, \qquad \forall \Phi \in C^0(\R^2)^2
\end{equation*}
which, by Area formula, can also be written as 
\begin{equation*}
\langle \Phi ,\bm \mu_\gamma \rangle =  \int_0^1 \Phi(\gamma(t)) \cdot \gamma'(t) \, dt .
\end{equation*}
Notice that this definition is well-posed, in the sense that it does not depend on the parametrization $\gamma$ of the curve. It is easy to see that $\dive \bm \mu_\gamma=0$ in the sense of distributions, as a consequence of the fact that $\gamma(0)=\gamma(1)$, so $\bm \mu_\gamma \in \mathcal J(\R^2)$.
\end{remark}

The following proposition establishes a functional analytic connection between $\mathcal J(\R^2)$ and $\FV(\R^2)$.

\begin{proposition}\label{prop:isomorphism}
	The map
	\begin{equation*}
	\begin{split}
	\nabla^\perp \colon \FV(\R^2) & \to \mathcal J(\R^2)  \\
	 f & \mapsto \bm \mu:= \nabla^\perp f = (- \partial_y f,\partial_x f )
	\end{split}
	\end{equation*}
	is an isometric isomorphism. 
\end{proposition}

\begin{proof}
	\emph{Well-posedness and linearity.} The map $\nabla^\perp$ is well-posed, because $\dive \nabla^\perp f = 0$ for any $f \in \FV(\R^2)$: indeed, for any test function $\phi \in C_c^\infty(\R^2)$ 
	\begin{equation*}
	\begin{split}
	\langle \dive \nabla^\perp f , \phi \rangle & = \int_{\R^2} \nabla \phi(z) \cdot d(\nabla^\perp f)(z) \\ 
	& = \int_{\R^2} (\partial_x \phi(z), \partial_y \phi(z)) \cdot d((- \partial_y f,\partial_x f ))(z) \\
	& = \int_{\R^2} \partial_y \partial_x \phi(z) f(z) \, dz - \int_{\R^2} \partial_x \partial_y \phi(z) f(z)  \, dz = 0.
	\end{split}
	\end{equation*} 
	Linearity of $\nabla^\perp$ is trivial. 
	
	\emph{Injectivity.} The kernel of $\nabla^\perp$ is given by the functions $f$ for which 
	\begin{equation*}
	\nabla^\perp f = 0
	\end{equation*}
	which means $f$ is constant in $\R^2$, in particular $f = 0$ in $\FV(\R^2)$: injectivity follows. 
	
	\emph{Surjectivity.} 
	Let us turn to prove surjectivity: pick $\mu \in \mathcal J(\R^2)$ and let $\{\rho_\e\}_{\e>0}$ be a standard family of mollifiers in $\R^2$. Set 
	\begin{equation*}
	\Phi_\e(x) := \bm \mu \ast \rho_\e(x) = \int_{\R^2} \rho_\e(x-y) d\bm \mu(y).
	\end{equation*}
	and observe that by standard facts $\Phi_\e \in C_c^\infty(\R^2;\R^2)$ with $\dive \Phi_\e=0$. By Poincar\'e Lemma, for every $\e>0$, there exists $f_\e \in C_c^\infty(\R^2)$ such that $\nabla^\perp f_\e = \Phi_\e$. Notice that for any $\e>0$
	\begin{equation*}
	V(f_\e) = \| \Phi_\e\|_1 \le \Vert \bm \mu \Vert_{\mathcal J}
	\end{equation*}
	hence $(f_\e)_{\e>0} \subset \FV(\R^2)$. By Theorem \ref{thm:347AFP} there exists $\{m_\e\}_{\e>0}\subset \R$ and a universal constant $\gamma>0$ such that 
	\begin{equation*}
	\Vert f_\e - m_\e \Vert_{L^{1^*}(\R^2)} \le \gamma V(f_\e) \le \gamma  \Vert \bm \mu \Vert_{\mathcal J}.
	\end{equation*}
	In particular, if we now fix any open $\Omega \Subset \R^2$, we have using H\"older inequality 
	\begin{equation*}
	\| f_\e-m_\e\|_{L^1(\Omega)} \le \L^d(\Omega)^{1/d} \| f_\e-m_\e\|_{L^{1^*}(\Omega)} \le \gamma\L^d(\Omega)^{1/d}  \Vert \bm \mu \Vert_{\mathcal J}.
	\end{equation*}
	On the other hand 
	\begin{equation*}
	V(f_\e-m_\e, \Omega) \le V(f_\e-m_\e) = V(f_\e)\le \Vert \bm \mu \Vert_{\mathcal J}
	\end{equation*}
	and hence we are in position to apply the Compactness Theorem \cite[Theorem 3.23]{AFP}: there exists a function $f \in L^1_\loc(\R^2)$ such that, up to a subsequence, $(f_\e-m_\e )\to f$ strongly in $L_{\loc}^1(\R^2)$ as $\e\to 0$. In particular, $f$ is also in $\FV(\R^2)$ by the l.s.c. of the total variation 
	\begin{equation*}
	V(f)\le \liminf_{\e \downarrow 0} V(f_\e) \le c  \Vert \bm \mu \Vert_{\mathcal J}.
	\end{equation*}
	It remains now to check that $\nabla^\perp f = \bm \mu$: indeed, for any smooth, compactly supported test function $\Psi \in C_c^\infty(\R^2, \R^2)$ it holds 
	\begin{equation*}
	\begin{split}
	\int_{\R^2} f(x) \dive  \Psi(x) \, dx & = \lim_{\e\to 0} \int_{\R^2} f_\e(x) \dive  \Psi(x) \, dx \\
	&  = \lim_{\e\to 0} \int_{\R^2}  \Psi(x) \cdot \Phi^\perp_\e(x)\,dx \\
	&  = \int_{\R^2}  \Psi(x) \, d\bm \mu^\perp(x) \\
	\end{split}
	\end{equation*}
	where in the last passage we have used that $\Phi_\e \rightharpoonup \bm \mu$ as $\e \to 0$ (see e.g. \cite[Thm. 2.2]{AFP}).
	
	\emph{$\nabla^\perp$ is an isometry.} It remains thus to show that $\nabla^\perp$ is an isometry: taken $f \in \FV(\R^2)$ by definition 
	\begin{equation*}
	\begin{split}
	\Vert f \Vert_{\FV} & = V(f,\R^2) = \sup \left\{ \int_{\R^2} f(x) \dive \Phi(x)\, dx: \, \Phi \in C_c^\infty(\R^2), \Vert \Phi \Vert_{\infty} \le 1 \right\}\\
	& = \sup \left\{ \langle \Phi, \nabla f \rangle: \,  \Phi \in C_c^\infty(\R^2), \Vert \Phi \Vert_{\infty} \le 1 \right\}\\
	& = \sup \left\{ \langle \Phi, \nabla^\perp f \rangle: \,  \Phi \in C_c^\infty(\R^2), \Vert \Phi \Vert_{\infty} \le 1 \right\}\\
	& = \sup \left\{ \int_{\R^2} \dive \Phi(x)\, d(\nabla^\perp f)(x):  \,  \Phi \in C_c^\infty(\R^2), \Vert \Phi \Vert_{\infty} \le 1 \right\}\\
	& = \Vert \nabla^\perp f \Vert_{\mathcal J}. 
	\end{split}
	\end{equation*}
\end{proof}

\section{Simple sets and closed curves}

Aim of this section is to give a detailed description of the extreme points of the unit ball of $\mathcal J$. Since $\nabla^\perp$ is an isometry we have 
\begin{equation*}
B_1^{\mathcal J} = \nabla^\perp (B_1^{Y})
\end{equation*}
and hence, by Remark \ref{rmk:extreme_isomor}, we have 
\begin{equation}\label{eq:extreme_points_sets}
\ext(B_1^{\mathcal J}) = \ext(\nabla^\perp (B_1^{Y})) = \nabla^\perp (\ext(B_1^{Y}) )= \left\{\sigma \frac{\nabla^\perp \1_E}{P(E)} : E \subset \R^2 \text{ simple set,  } P(E)>0 \text{ and } \sigma \in \{\pm 1\}\right\}.
\end{equation}
Let us introduce the following notation: 
\begin{equation*}
	\Gamma := \left\{ \gamma \colon [0,1] \to \R^2:  \, \text{Lipschitz on } [0,1], \text{ injective on } [0,1) \text{ and }\gamma(0)=\gamma(1)  \right\}. 
\end{equation*}
For any $\gamma \in \Gamma$, we define its length to be $\ell(\gamma):=\int_0^1 \vert \gamma'(t)\vert \, dt \in(0,+\infty)$. Notice, in particular, that any $\gamma \in \Gamma$ induces a rectifiable Jordan curve $C := \gamma([0,1])$, and viceversa every rectifiable Jordan curve can be parametrized by some $\gamma \in \Gamma$. Being a subset of $(\Lip[0,1])^2$ the space $\Gamma$ can be thought as a normed space, being the norm the (restriction of the) uniform one $\Vert  \cdot \Vert_{\infty}$.

We now want to prove the following proposition.

\begin{proposition}\label{prop:extreme_points_of_H}
	The following equality holds true: 
	\begin{equation*}
	\ext(B_1^{\mathcal J}) =\left\{ \frac{1}{\ell(\gamma)} \bm \mu_\gamma: \, \gamma\in \Gamma \right\} . 
	\end{equation*}
\end{proposition}
 
\begin{proof} Let $\mu \in 	\ext(B_1^{\mathcal J})$. From \eqref{eq:extreme_points_sets} we have that $\mu = \sigma \frac{1}{P(E)} \nabla^\perp \1_E$ for some simple set $E \subset \R^2$ with $P(E)>0$ and $\sigma \in {\pm 1}$. From Theorem \ref{thm:simple-boundary}, the essential boundary $\partial^e E$, is equivalent, up to an $\H^1$-negligible set, to a rectifiable Jordan curve. Using Theorem \ref{thm:de_giorgi_federer}, we can conclude that also $\mathcal F E$ can be parametrized by some Jordan curve, which can be taken to be Lipschitz (see \cite[Lemma 3]{ACMM}). All in all, we have that there exists $\gamma \in \Gamma$ such that $\gamma([0,1]) = \mathcal F E$, up to a $\H^1$-null set. 

On the one hand, by De Giorgi's Theorem \ref{thm:de_giorgi}, for $\mathscr H^1$-a.e. $x \in \mathcal F E$ we have 
\begin{equation}\label{eq:tan_de_giorgi}
\Tan(\mathcal F E, x) = \mathsf{span} (\nu^\perp_E(x)) 
\end{equation}
where $\nu_E(x)$ is the generalized inner normal to $E$ and $\mathsf{span} (\nu^\perp_E(x))$ denotes the orthogonal line to $\nu_E(x)$.
 
On the other hand, since $\mathcal F E = \gamma([0,1])$ we have using Proposition \ref{prop:tang_smooth} 
\begin{equation}\label{eq:tan_non_de_giorgi}
\Tan(\gamma([0,1]), x) = \mathsf{span}(\gamma'(\gamma^{-1}(x))). 
\end{equation}
Since the approximate tangent space is a one dimensional vector space and since $\nu_E(x)$ is unit vector for $\H^1$-almost every $x \in \mathcal FE$,   equalities \eqref{eq:tan_de_giorgi} and \eqref{eq:tan_non_de_giorgi} force that for $\H^1$-a.e. $x \in \mathcal FE$
\begin{equation}\label{eq:equality_x}
\nu^\perp_E(x) = \sigma(x) \frac{\gamma'(\gamma^{-1}(x))} {|\gamma'(\gamma^{-1}(x)), |} \qquad \text{ for }\sigma(x) \in \{\pm 1\}. 
\end{equation}
This means that the vector $\nu^\perp_E(x)$ is tangent to the curve $\gamma$ at the point $\gamma(\gamma^{-1}(x))$ for $\H^1$-a.e. $x \in \gamma([0,1])$. Since $\dive(\nu^\perp_E \H^1 \rest_{\gamma([0,1])})=0$ we can apply
\cite[Theorem 4.9]{BG}, obtaining that
\begin{equation*}
\exists \bar\sigma \in \{\pm 1\}: \qquad  \nu_E^\perp(\gamma(t)) = \bar \sigma \cdot \frac{\gamma'(t)}{|\gamma'(t)|} \qquad \text{ for } \L^1\text{-a.e. } t \in [0,1].
\end{equation*}
 Reversing the parametrization of $\gamma$, if necessary, one can achieve that $\bar \sigma = 1$. Then for any test function $\Phi \in C_c^\infty(\R^2;\R^2)$, using Area Formula, we obtain
\begin{equation}\label{eq:long_perim_curves}
\begin{split}
\langle \mu, \Phi \rangle & = \left \langle \frac{1}{P(E)} \nabla^\perp \1_E , \Phi \right \rangle \\
& = \left \langle \frac{1}{P(E)}\nu_E^\perp \H^1\rest_{\mathcal FE}, \Phi \right \rangle \\
& = \frac{1}{P(E)} \int_{\R^2}  \Phi(x) \cdot  \nu_E^\perp(x) \, d\H^1 \rest_{\mathcal FE}(x) \\ 
& = \frac{1}{P(E)}  \int_{\R^2} \Phi(x) \cdot \frac{ \gamma'(\gamma^{-1}(x))}{|\gamma'(\gamma^{-1}(x))|} \, d\H^1 \rest_{\gamma([0,1])}(x) \\ 
& =  \frac{1}{P(E)} \int_0^1 \Phi(\gamma(t)) \cdot  \frac{\gamma'(t)}{|\gamma'(t)|} |\gamma'(t)| \, dt\\
& =  \frac{1}{\ell(\gamma)} \int_0^1 \Phi(\gamma(t)) \cdot  \gamma'(t)\, dt \\ 
& =\left \langle  \frac{1}{\ell(\gamma)}\bm \mu_{\gamma}, \Phi \right \rangle 
\end{split}
\end{equation}
where we have also used the fact that $P(E) = V(\1_E) =\|\nu_E \H^1\rest_{\mathcal FE}\|_{\mathscr M} =\H^1(\gamma([0,1])) = \ell(\gamma)$ (which also follows from Area Formula). 

Thus we have shown that any extreme point $\mu$ of $B_1^{\mathcal J}$ has necessarily the form $\frac{1}{\ell(\gamma)}\bm \mu_{\gamma}$. The converse implication, namely that normalized measures $\mu_{\gamma}$ are extreme, follows immediately from the second part of Theorem \ref{thm:simple-boundary}: any $\gamma \in \Gamma$ induces a rectifiable Jordan curve $C := \gamma([0,1])$, hence $\mathsf{int}(\Gamma) =:E$ is a simple set by Theorem \ref{thm:simple-boundary}. Extremality the follows from \eqref{eq:extreme_points_sets}, noticing that 
\begin{equation*}
\frac{1}{\ell(\gamma)}\bm \mu_{\gamma} = \frac{1}{P(E)} \nabla^\perp \1_E
\end{equation*}
as above, and the proof is thus complete. 
\end{proof}

\section{Measures as superposition of curves I: a proof using Choquet's Theory}

In this section we prove \hyperref[t-main]{Main Theorem} with $\rho=0$:

\begin{theorem}\label{thm:main}
	Let $\bm \mu \in \mathcal J(\R^2)$, where $\mathcal J(\R^2)$ is as in Definition \ref{def:solenoids}. Then there exists a $\sigma$-finite, non-negative measure $\eta \in \mathscr M_+(\Gamma)$ such that \eqref{e-mu-decomp} and \eqref{e-|mu|-decomp} hold.
\end{theorem}

Consider the maps $\mathfrak p  \colon \Gamma \to \mathcal J(\R^d)$ and $F\colon \mathfrak p(\Gamma) \to \mathcal J(\R^d)$ defined by
\begin{equation}\label{e-maps-p-and-F}
\mathfrak p(\gamma) := \bm \mu_{\gamma}, \qquad F(\bm \nu) := 
\begin{cases}
\frac{\bm \nu}{\|\bm \nu\|}, & \bm \nu \ne 0,\\
0, & \bm \nu = 0
\end{cases}
\end{equation}
For any $m\in \N$ let
\begin{equation*}
\Gamma_m := \left\{\gamma \in \Gamma: |\gamma(0)| + \|\gamma'\|_{\infty} \le m \right\}.
\end{equation*}
In view of Arzel\`a-Ascoli Theorem $\Gamma_m$ is a compact subset of $\Gamma$ (with respect to the topology of the uniform convergence).

The lemma below works in any dimension $d$ (not only $d=2$). 
\begin{lemma}\label{lemma:borel}
	The maps $\mathfrak p$ and $F$ defined in \eqref{e-maps-p-and-F} have the following properties:
	\begin{enumerate}
	\item For any $m\in \N$ the map $\mathfrak p  \colon \Gamma_m \to \mathcal J(\R^d)$ defined in \eqref{e-maps-p-and-F}
	is continuous (with respect to uniform topology on $\Gamma_m$ and weak-star topology on $\mathcal J(\R^d)$).
	\item The map $F\colon \mathcal J(\R^d) \to \mathcal J(\R^d)$ is Borel.
	\item The sets $\mathfrak p(\Gamma)$ and $F(\mathfrak p(\Gamma))$ are Borel.
	\item The $F\colon \mathfrak p(\Gamma) \to F(\mathfrak p(\Gamma))$ has Borel inverse $F^{-1}$.
	\end{enumerate}
\end{lemma}

\begin{proof}
It is sufficient to verify sequential continuity of $\mathfrak p$. Let $(\gamma_k)_{k \in \N} \subset \Gamma_m$ be a sequence with $\gamma_n \to \gamma$ for a certain $\gamma \in \Gamma_m$. Let us show that $\bm \mu_{\gamma_n} \stackrel{*}{\rightharpoonup} \bm \mu_{\gamma}$ first in the sense of distributions: let $\Phi \in C_c^\infty(\R^d;\R^d)$.
	Then
	\begin{equation*}
	\begin{split}
	\vert \langle \bm \mu_{\gamma_n}, \Phi \rangle -  \langle \bm \mu_{\gamma}, \Phi \rangle \vert & =\left \vert  \int_{0}^1 \Phi(\gamma_n(t)) \cdot \gamma'_n(t) dt - \int_{0}^1 \Phi(\gamma(t)) \cdot \gamma'(t) dt \right \vert \\
	& = \left \vert  \int_0^1 \bigl( \Phi(\gamma_n(t)) - \Phi(\gamma(t)) \bigr)  \cdot \gamma'_n(t) dt - \int_0^1 \Phi(\gamma(t)) \cdot \left( \gamma'(t) - \gamma'_n(t) \right) dt \right \vert \\
	& \le m\int_{0}^1 \vert \Phi(\gamma_n(t)) - \Phi(\gamma(t)) \vert \, dt + \left \vert \int_{0}^1 \frac{d}{dt}  \Phi(\gamma(t)) \cdot \left( \gamma(t) - \gamma_n(t) \right) \, dt \right\vert \\
	& \le m\int_{0}^1 \vert \Phi(\gamma_n(t)) - \Phi(\gamma(t)) \vert \, dt + \Vert \nabla \Phi \Vert_{\infty} \int_0^1  |\gamma(t) - \gamma_n(t) | \, dt \to 0 
	\end{split}
	\end{equation*}
	as $n \to +\infty$. Moreover $\sup_{n\in \N} \|\bm \mu_{\gamma_n}\| \le m$.
	Hence the functionals $\bm \mu_{\gamma_n} \in C_0(\R^d;\R^d)^*$ are uniformly bounded and converge to $\bm \mu$ pointwise on the set $C_c^\infty(\R^d;\R^d)$ which is dense in $C_0(\R^d;\R^d)$. Therefore $\bm \mu_{\gamma_n} \stackrel{*}{\rightharpoonup} \bm \mu$ as $n\to \infty$.

	Since for any $m\in \N$ the set $\mathfrak p(\Gamma_m)$ is compact (being an image of a compact under a contunuous map), the set $\mathfrak p(\Gamma) = \bigcup_{m\in \N} \mathfrak p(\Gamma)$ is Borel.

For any $\Phi\in C_0(\R^d;\R^d)$ the map $\bm \nu \mapsto \frac{\langle \bm \nu, \Phi \rangle}{\|\bm \nu\|}$ is Borel. Indeed, the numerator is a continuous function of $\bm \nu$ and the denominator is lower semicontinuous (hence Borel).
Therefore $F$ is a Borel map from $\mathscr M(\R^d;\R^d)$ to $\mathscr M(\R^d;\R^d)$ (with respect to weak-star topologies). 
Since for every $m\in \N$ the set $\mathfrak p(\Gamma_m)$ is contained in a closed ball of $\mathscr M(\R^d; \R^d)$ (note that this ball is Polish with respect to weak* topology) and $F$ is injective on $\mathfrak p(\Gamma_m)$, it follows that $F(\mathfrak p(\Gamma_m))$ is Borel
(see e.g. \cite{bogachev}, Theorem 6.8.6). Therefore $F(\mathfrak p(\Gamma)) = \bigcup_{m=1}^\infty F(\mathfrak p(\Gamma_m))$ is Borel. Similarly, the image of any Borel subset of $\mathfrak p(\Gamma)$ under $F$ is Borel, and by injectivity of $F$ on $\mathfrak p(\Gamma)$ this means that $F\colon \mathfrak p(\Gamma) \to F(\mathfrak p(\Gamma))$ has Borel inverse.
\end{proof}

\begin{lemma}\label{l-reparam}
Suppose that $\bm \mu \in \mathcal J(\R^d)$ and there exists a finite measure $\xi \in \mathscr M_+(\mathcal J(\R^d))$ concentrated on $F(\mathfrak p(\Gamma))$ such that
\begin{equation}\label{e-mu-decomposition-on-J}
\bm \mu = \int_{\mathcal J(\R^d)} \bm \nu \, d\xi(\bm \nu), \qquad
|\bm \mu| = \int_{\mathcal J(\R^d)} |\bm \nu| \, d\xi(\bm \nu)
\end{equation}
Then there exists $\sigma$-finite $\eta \in \mathscr M_+(\Gamma)$ such that 
\eqref{e-mu-decomp} and \eqref{e-|mu|-decomp} hold for $\bm \mu$ and $\eta$.
\end{lemma}

\begin{proof}
By Lemma~\ref{lemma:borel} the map $F \colon \mathfrak p(\Gamma) \to F(\mathfrak p (\Gamma))$ has Borel inverse $F^{-1}$ hence we can change variables using the map $F^{-1}$:

	\begin{equation*}
	\int_{F(\mathfrak p (\Gamma))} y \, d \xi(y) = \int_{F(\mathfrak p (\Gamma))} F(F^{-1}(y)) \, d \xi(y)
	= \int_{\mathfrak p(\Gamma)} F(\bm \nu) \, d (F^{-1}_\#\xi)(\bm \nu)
	= \int_{\mathfrak p(\Gamma)} \bm \nu \, d \hat\xi(\bm \nu).
	\end{equation*}
	where $\hat\xi$ denotes the measure on $\mathscr M(\R^d;\R^d)$ defined by
	\begin{equation*}
	\hat\xi(A):= \int_{A\setminus \{0\}} \frac{1}{\|\bm \nu\|}\, d (F^{-1}_\#\xi)(\bm \nu),
	\end{equation*}
	$A\subset \mathscr M(\R^d;\R^d)$ being an arbitrary Borel subset
	(clearly $\hat\xi$ is concentrated on $\mathfrak p(\Gamma)$).

	Since $\mathfrak p(\Gamma) = \bigcup_{m\in \N} \mathfrak p(\Gamma_m)$ we can write
	$\hat \xi$ as a sum of its restrictions $\hat\xi_m$ on the sets $\mathfrak p(\Gamma_{m+1}) \setminus \mathfrak p(\Gamma_{m})$, where $m\in \N$.

	By Lemma~\ref{lemma:borel} the map $\mathfrak p\colon \Gamma_m \to \mathscr M(\R^d;\R^d)$ is continous and the set $\Gamma_m$ is compact, hence there exists a Borel set $B_m \subset \Gamma_m$ such that the restriction of $\mathfrak p$ to $B_m$ is injective and $\mathfrak p(B_m) = \mathfrak p(\Gamma_m)$ (see e.g. \cite[Theorem 6.9.7]{bogachev}). Therefore the inverse map $\mathfrak q_m \colon \mathfrak p(\Gamma_m) \to \Gamma_m$ is Borel.
	Now we change variables using $\mathfrak q_m$:
	\begin{equation*}
	\int_{\mathfrak p(\Gamma)} \bm \nu \, d \hat\xi_m(\bm \nu)
	= \int_{\mathfrak p(\Gamma_m)} \mathfrak p(\mathfrak q_m(\bm \nu)) \, d \hat\xi_m(\bm \nu)
	= \int_{\Gamma_m} \mathfrak p(\gamma) \, d (\mathfrak (q_m)_\# \hat\xi_m)(\gamma)
	= \int_{\Gamma} \bm \mu_{\gamma} \, d (\eta_m)(\gamma)
	\end{equation*}
	where $\eta_m := (\mathfrak q_m)_\# \hat\xi_m$. Denoting $\eta := \sum_{m=1}^\infty \eta_m$ we ultimately obtain
	\begin{equation*}
  \bm \mu  = \int_{\Gamma}\bm \mu_{\gamma}  \, d \eta(\gamma). 
	\end{equation*}

	Equality holds for total variations as well: indeed, by triangle inequality
	\begin{equation*}
  |\bm \mu| \le  \int_{\Gamma} \left\vert\bm \mu_{\gamma} \right\vert \, d \eta(\gamma) \qquad \text{ as measures on } \R^d. 
	\end{equation*}	
	If the inequality above were strict, then by evaluating it on the whole $\R^d$ we would get a contradiction:
	\begin{equation*}
	\|\bm \mu\| = |\bm \mu|(\R^d) <  \int_{\Gamma} \left\vert \bm \mu_{\gamma} \right\vert(\R^d) \, d \eta(\gamma) =  \int_{\mathcal J(\R^d)} |\bm \nu|(\R^d) \, d\xi(\bm \nu) = \|\bm \mu\|.
	\end{equation*}
	%\textcolor{blue}{(Here we have performed the same changes of variables as above.)}

	Since $\|\bm \mu\| = \int \|\bm \mu_\gamma\| \, d\eta(\gamma)$ and for any $k\in \N$ the set $\{\gamma \in \Gamma : \|\bm \mu_\gamma\| > k^{-1}\}$ is Borel, it is clear that $\eta$ is $\sigma$-finite.
\end{proof}

We are now ready to prove the Main Theorem. 
\begin{proof}[Proof of Theorem \ref{thm:main}]
By Proposition \ref{prop:extreme_points_of_H} we have
\begin{equation*}
\ext(B_1^{\mathcal J}) = \left\{ \frac{1}{\ell(\gamma)} \bm \mu_\gamma: \, \gamma\in \Gamma \right\}
\subset F(\mathfrak p (\Gamma)).
\end{equation*}
By Remark \ref{rmk:extreme_borel} the set $\ext(B_1^{\mathcal J})$ is Borel.

Let $0 \ne \bm \mu \in \mathcal J(\R^2)$ and consider the normalized measure 
	\begin{equation*}
	\frac{\bm \mu}{\|\bm \mu\|}  \in B_1^{\mathcal J}.
	\end{equation*} 
By Choquet's Theorem \ref{thm:Choquet} there exists a Borel probability measure $\pi \in \mathscr P(\ext B_1^{\mathcal J})$ such that 
	\begin{equation*}
	\frac{\bm \mu}{\|\bm \mu\|}  = \int_{\ext  B_1^{\mathcal J}} y \, d \pi(y)
	\end{equation*}
	the integral being understood in the sense of Definition \ref{def:vector_integration}. 
	By triangle inequality we deduce from the equality above that 
	\begin{equation*}
	|\bm \mu| \le \|\bm \mu\| \int_{\ext  B_1^{\mathcal J}} |y| \, d \pi(y).
	\end{equation*}
	 Note that the latter inequality is in fact an equality, since otherwise by evaluating it on $\R^2$ we would get a contradiction:
	\begin{equation*}
	\|\bm \mu\| < \|\bm \mu\| \int_{\ext  B_1^{\mathcal J}} \|y\| \, d\pi(y) = \|\bm \mu\|.
	\end{equation*}
	In order pass to integration over $\Gamma$ instead of $\ext(B_1^{\mathcal J}) \subset \mathscr M(\R^2;\R^2)$, it remains to change
	variables by applying Lemma~\ref{l-reparam} with $\xi := \|\bm \mu\| \pi$.
	This concludes the proof.
	\end{proof}

Note that the elements of $\Gamma$ are not necessarily simple. However since the measure $\pi$ is concentrated on a set of measures induced by simple curves, it is easy to see from the proof of Theorem~\ref{thm:main} that for $\eta$-a.e. $\gamma\in \Gamma$ there exists a simple $\tilde{\gamma} \in \Gamma$ such that $\bm \mu_\gamma = \bm \mu_{\tilde{\gamma}}$.

\section{Measures as superposition of curves II: a proof using decomposition of \texorpdfstring{$\FV$}{FV} functions}

In this section, we present an alternative proof of Theorem \ref{thm:main}. This proof does not rely on Choquet's Theory, but it is based on the following decomposition result for $\FV$ functions.

\begin{theorem}\label{thm:BT-in_text}
	Let $f \in \FV(\R^d)$. There exists an at most countable family $\{f_i\}_{i \in I} \subset \FV(\R^d)$ of monotone functions such that the series 
	\begin{equation*}	
	\sum_{i \in I} f_i 
	\end{equation*}
	converges as an element of $\FV(\R^d)$ and  
	\begin{equation}\label{eq:BT}
	f=\sum_{i \in I} f_i  \qquad \text { with } \qquad  \|f\|_{\FV} = \sum_{i \in I} \|f_i\|_{\FV}.
	\end{equation}
\end{theorem}

For the definition of monotone function and for a proof of Theorem \ref{thm:BT-in_text} as well we refer the reader to the Appendix \ref{app}.

\begin{proof}[Proof of Theorem \ref{thm:main} using Theorem \ref{thm:BT-in_text}] 
	Let $\bm \mu \in \mathcal J(\R^2)$ and let $H \in \FV(\R^2)$ be the function such that $\bm \mu = \nabla^\perp H$, whose existence and uniqueness are granted by Proposition \ref{prop:isomorphism}.
	
	\emph{Case 1.} Suppose first that $H$ is monotone. % and signed (w.l.o.g. $H \ge 0$) and assume $V(H) \le 1$. 
	Let $E_t:=\{H>t\}$. Since the function $H$ lies in $\FV(\R^2)$ we have $H \in L^{1^*}(\R^2)$: using Chebyshev's inequality this integrability property implies that for a.e. $t \in \R$ it holds $\L^2(E_t)<\infty$. Combined with Coarea Formula, this observation yields the existence of a set $N \subset \R$ such that $\L^1(N)=0$ and $E_t$ has finite measure and finite perimeter for every $t \in \R\setminus N$.
	Consider now the function $f \colon \R\to \mathcal J(\R^2)$ defined by 
	\begin{equation*}
	f(t) := \begin{cases}
	\frac{\nabla^\perp \1_{E_t} }{P(E_t)} & \text{ if } \L^2(E_t) >0 \text{ and }  t \notin N \text{ with } P(E_t)>0 \\
	0 & \text{ otherwise}
	\end{cases}
	\end{equation*}
	and the measure $\rho \in \mathscr M_+(\R)$ 
	\begin{equation*}
	\rho(dt) := P(E_t) \L^1(dt). 
	\end{equation*}
	By Coarea Formula, we have 
	\begin{equation}\label{eq:coarea_f_rho}
	\nabla^\perp H = \int_{\R} f(t) \, d\rho(t)
	\end{equation}
	and Fubini's Theorem further ensures that $f$ is a measurable measure-valued map (see \cite[Def. 2.25]{AFP}). 
	In particular, from \eqref{eq:coarea_f_rho} we deduce for any $\Psi \in C_c(\R^2)^2$ 
	\begin{equation}\label{eq:test}
	\langle \nabla^\perp H, \Psi \rangle = \int_{\R} \langle f(t) , \Psi \rangle \, d\rho(t) = \int_{f(\R)} \langle y, \Psi \rangle d\eta(y)
	\end{equation}
	where we have set 
	\begin{equation*}
	\xi := f_{\#}\rho.
	\end{equation*}
	From \eqref{eq:test} and from the arbitrariness of test function $\Psi$, we deduce the sought formula 
	\begin{equation*}
	\nabla^\perp H = \int_{\mathcal J(\R^2)} y\, d\eta(y). 
	\end{equation*}
	Observe that, for every $t \in \R \setminus N$ such that $\L^2(E_t) >0 $ the computations in \eqref{eq:long_perim_curves} yield the equality 
		\begin{equation*}
		f(t) = \frac{\bm \mu_{\gamma_t}}{P(E_t)}
		\end{equation*}
	where $\gamma_t$ is the parametrization of $\partial^\star E_t$ given by Theorem \ref{thm:simple-boundary}. 
	Thus, by the very definition of $f$, the measure $\xi$ is concentrated on $F(\mathfrak p(\Gamma))$ (see page~\pageref{e-maps-p-and-F}). Moreover,
	\begin{equation}\label{eq:total_variation_ineq}
	\| \xi \| = \| f_{\#}\rho  \| = \| \rho  \| = \int_\R  P(E_t) \L^1(dt) = \| \nabla^\perp H \|.
	\end{equation}
		
	\emph{Case 2.}  If $H$ is not monotone, apply Theorem \ref{thm:BT-in_text} to the function $H$ and let $\{H_i\}_{i \in I}$ be at most countable family of monotone functions satisfying \eqref{eq:BT} (without loss of generality we may assume that $I=\N$). Let $\xi_i$ be a measure representing $\nabla^\perp H_i$ (obtained as in Case 1, since $H_i$ is monotone). Then it is easy to see that
	\begin{equation*}
	\xi:= \sum_{i=1}^{\infty} \xi_i
	\end{equation*}
	defines a measure on $\mathcal J(\R^2)$, as the series converges strongly: indeed, by \eqref{eq:total_variation_ineq} and \eqref{eq:BT} we get 
	\begin{equation*}
	\sum_{i=1}^{\infty} \| \xi_i \| = \sum_{i=1}^{\infty} \|\nabla^\perp H_i \| = \| \nabla^\perp H \| < \infty. 
	\end{equation*}
	Since the series above converge, we can pass to the limit as $n\to \infty$ in the equalities
	\begin{equation*}
	\sum_{i=1}^n \nabla^\perp H_i = \int_{\mathcal J(\R)} y\, d ({\textstyle\sum_{i=1}^n \xi_i})(y),
	\qquad
	\sum_{i=1}^n |\nabla^\perp H_i| = \int_{\mathcal J(\R)} |y|\, d ({\textstyle\sum_{i=1}^n \xi_i})(y).
	\end{equation*}
	We thus see that $\bm \mu$ and $\xi$ defined above satisfy \eqref{e-mu-decomposition-on-J}.
	It remains to change variables using Lemma~\ref{l-reparam}.
\end{proof}

\section{Linear rigidity for vector-valued measures}

In this section we give a proof of Theorem~\ref{t-linear-rigidity}, which is inspired by (and generalizes) one of the results from \cite{LeoSar17} (see Theorem~1.2 therein).

\begin{lemma}\label{lemma:orientation}
	Let $\bm \mu \in \mathscr M(\R^d;\R^d)$ and consider its polar decomposition $\bm \mu = \tau |\bm \mu|$. Suppose that there exists $\eta \in \mathscr M_+(\Gamma)$ such that \eqref{e-mu-decomp} and \eqref{e-|mu|-decomp} hold. Then for $\eta$-a.e. $\gamma\in \Gamma$
	\begin{equation*}
	\gamma' (t) =\tau(\gamma(t)) |\gamma'(t)|
	\end{equation*}
	for a.e. $t \in [0,1]$.
\end{lemma}

\begin{proof} Since $|\bm \mu|$ is a finite measure and $C_c(\R^d)$ is dense in $L^1(|\bm \mu|)$, we can use $\tau$	as a test function in the distributional formulation of \eqref{e-mu-decomp}, obtaining  
\begin{equation*}
\int_{\R^d} \tau \cdot d\bm \mu =\int_{\Gamma} \bm  \tau \cdot \bm \mu_{\gamma} \, d\eta(\gamma) = 
\int_{\Gamma} \int_0^1 \tau (\gamma(t)) \cdot  \gamma'(t) \, dt \, d\eta(\gamma). 
\end{equation*}
On the other hand, 
\begin{equation*}
\int_{\R^d} \tau \cdot d\bm \mu = |\bm \mu| (\R^d)=
\int_{\Gamma}  \vert \bm \mu_{\gamma}  \vert \, d\eta(\gamma) = \int_{\Gamma} \int_0^1  |\gamma'(t)| \, dt \, d\eta(\gamma). 
\end{equation*}
Therefore
\begin{equation}\label{e-tmp-int-zero}
 \int_{\Gamma} \int_0^1 \Big( \tau (\gamma(t)) \cdot  \gamma'(t) -  |\gamma'(t)| \Big) \, dt \, d\eta(\gamma) = 0. 
\end{equation}
The integrand is non-positive, since
\begin{equation*}
\tau (\gamma(t)) \cdot \gamma'(t)  \le| \tau(\gamma(t))| \cdot |\gamma'(t) | =|\gamma'(t)|.
\end{equation*}
hence by \eqref{e-tmp-int-zero} for $\eta$-a.e. $\gamma$
\begin{equation*}
\tau (\gamma(t)) \cdot \gamma'(t)  = |\gamma'(t)|
\end{equation*}
for a.e. $t \in [0,1]$.
\end{proof}

Recall the following definition: $\bm \sigma\in \mathscr M(\R^d;\R^d)$ is called a \emph{subcurrent of $\bm \mu \in \mathscr M(\R^d;\R^d)$} if $$\|\bm \mu\| = \|\bm \mu - \bm \sigma\| + \|\bm \sigma\|.$$

\begin{proposition}\label{prop-generic-subcurrent-form}
Let $\bm \mu \in \mathscr M(\R^d;\R^d)$.
Then $\bm \sigma\in \mathscr M(\R^d;\R^d)$ is a subcurrent of $\bm \mu$ if and only if
\begin{equation*}
\bm \sigma = g \bm \mu
\end{equation*}
where $g \in L^1(|\bm \mu|)$ satisfies $0\le g(x) \le 1$
for $|\bm \mu|$-a.e. $x\in \R^d$.
\end{proposition}

\begin{proof}
\emph{Sufficiency.} If $g \in L^1(|\bm \mu|)$ satisfies $0\le g(x) \le 1$ and $\bm \sigma = g \bm \mu$ then
\begin{equation*}
|\bm \mu - \bm \sigma| + |\bm \sigma| = (1 - g) |\bm \mu| + g |\bm \mu| = |\bm \mu|,
\end{equation*}
and it remains to evaluate the equality above on $\R^d$.

\emph{Necessity.}
By Radon--Nikodym theorem there exist mutually singular $\bm \sigma^a, \bm \sigma^s \in \mathscr M(\R^d;\R^d)$ such that $\bm \sigma^a \ll |\bm \mu|$, $\bm \sigma^s \perp |\bm \mu|$ and $\bm \sigma = \bm \sigma^a + \bm \sigma^s$. Then by definition of subcurrent
\begin{equation*}
\begin{aligned}
\| \bm \mu \| &= \| \bm \mu - \bm \sigma \| + \| \bm \sigma \| \\
              &= \| \bm \mu - \bm \sigma^a \| + \| \bm \sigma^a \| + 2 \| \bm \sigma^s \| \\
              &\ge \| \bm \mu \| + 2 \| \bm \sigma^s \|
\end{aligned}
\end{equation*}
by triangle inequality, hence $\|\bm \sigma^s\|=0$.
Therefore $\bm \sigma = \theta |\bm \mu|$ and $\bm \mu = \tau |\bm \mu|$ for some $\theta, \tau \in L^1(|\bm \mu|;\R^d)$ (by polar decomposition). Writing again the definition of subcurrent we obtain
\begin{equation*}
\int (|\tau| - |\tau - \theta| - |\theta|) \, d |\bm \mu| = 0,
\end{equation*}
which implies (in view of triangle inequality) that 
\begin{equation}\label{e-wrk2-int}
|\tau(x)| - |\tau(x) - \theta(x)| - |\theta(x)| = 0
\end{equation}
for $|\bm \mu|$-a.e. $x\in \R^d$. In particular, for $|\bm \mu|$-a.e. $x\in \R^d$ if $\tau(x) = 0$ then $\theta(x) = 0$.
Since vectors $a=\theta(x)$ and $b = \tau(x) - \theta(x)$ with $a\ne 0$ satisfy $|a+b| - |a| - |b| = 0$ if and only if $b = |b| \frac{a}{|a|}$, 
we conclude that there exists $g=g(x)\in \R$ such that $\theta(x) = g(x) \tau(x)$.
Substituting this into \eqref{e-wrk2-int} we conclude that $0\le g(x) \le 1$ for $|\bm \mu|$-a.e. $x\in \R^d$.
\end{proof}

\begin{corollary}\label{cor:acyclic}
Suppose that $\bm \nu \in \mathscr M_\loc(\R^d;\R^d)$ has polar decomposition $\bm \nu = \tau |\bm \nu|$.
If $\tau_1(x) > 0$ for~$|\bm \nu|$-a.e. $x\in \R^d$ then $\bm \nu$ is acyclic.
\end{corollary}

\begin{proof}
For any $\bm \mu \in \mathscr M(\R^d;\R^d)$ satisfying \eqref{e-div-mu-eq-rho} with some $\rho\in \mathscr M(\R^d)$ the distributional formulation of \eqref{e-div-mu-eq-rho} holds for any test function $\fhi\in C^\infty(\R^d)$ such that $\|\fhi_\infty\| + \|\nabla \fhi\|_\infty < \infty$. In order to prove this it is sufficient to consider $\omega\in C_c^\infty(\R^d)$ such that $\omega(x) = 1$ if $|x|\le 1$ and $\omega(x) = 0$ if $|x|\ge 2$ and pass to the limit in
\begin{equation*}
-\int_{\R^d} \nabla(\fhi(x) \omega(R^{-1} x)) \cdot d\bm \mu(x) = \int_{\R^d} \fhi(x) \omega(R^{-1} x) \, d \rho(x)
\end{equation*}
as $R\to \infty$ using dominated convergence theorem.

In particular, if $\bm \sigma$ is a cycle of $\bm \nu$ then by Proposition~\ref{prop-generic-subcurrent-form} there exists $g \in L^1(|\bm \nu|)$ such that $0\le g(x) \le 1$
for $|\bm \nu|$-a.e. $x\in \R^d$ and $\bm \sigma = g \bm \nu$.
Writing the distributional formulation of $\dive (\bm \sigma) = 0$ with the test function $\fhi(x) = \operatorname{\mathrm{atan}}(x_1)$ we get
\begin{equation*}
- \int_{\R^d} \frac{g(x) \tau_1(x)}{1 + x_1^2} \, d |\bm \nu|(x) = 0
\end{equation*}
hence $g(x)=0$ for $|\bm \nu|$-a.e. $x\in \R^d$. Therefore $\bm \sigma = 0$ is the only cycle of $\bm \nu$.
\end{proof}

\begin{proof}[Proof of Theorem~\ref{t-linear-rigidity}]
Suppose that $\bm \nu \in \mathscr M_\loc(\R^d;\R^d)$ satisfies (i)--(iii) from Definition~\ref{d-linear-rigidity}.

Let $\omega\in C_c^\infty(\R^d)$ be a nonnegative function such that $\omega(x) = 1$ if $|x|\le 1$ and $\omega(x) = 0$ if $|x|\ge 2$. Let~$h>0$ and $r>0$ and let $R>0$ be such that $r^2 + h^2 < R^2$ and $r + c^{-1} h < R$, where $c>0$ is the constant from Definition~\ref{d-linear-rigidity}.

Let $\bm \mu = f \cdot \bm \nu$ where  $f(x) = \omega(x/R)$. Clearly $\dive \bm \mu$ belongs to $\mathscr M(\R^d)$ and is concentrated on 
\begin{equation*}
A:=\{x \in \R^d : x_d \ge 0, R \le |x| \le 2 R\}.
\end{equation*}
Moreover, $\bm \mu$ is acyclic by Corollary~\ref{cor:acyclic}.

For any $x\in \R^d$ let $x_o:=(x_1, \ldots, x_{d-1})$. Let
\begin{equation}\label{eq:def_T}
T:= \{y\in \R^d : 0 < y_d < h, \; |y_o| < r + c^{-1}(h - y_d) \}.
\end{equation}

Let $\eta\in \mathscr M_+(\Gamma)$ be given by Theorem~\ref{t-acyclyc-decomposition} applied to $\bm \mu$ (in particular \eqref{e-mu-decomp}--\eqref{e-|div-mu|-decomp} hold).
Let 
\begin{equation*}
\Gamma_T := \{\gamma \in \Gamma : |\bm \mu_\gamma|(T) > 0\}.
\end{equation*}

Note that $\bm \mu = \tau f |\bm \nu| = \tau |\bm \mu|$ is the polar decomposition of $\bm \mu$. Hence by Lemma~\ref{lemma:orientation} for $\eta$-a.e. $\gamma \in \Gamma_T$ for a.e. $z\in[0,1]$ we have
\begin{equation*}
\gamma'(z) = \tau(\gamma(z)) |\gamma'(z)|
\end{equation*}
Writing this equation for $\gamma_d$ and $\gamma_o$ separately and using condition (iii) from Definition~\ref{d-linear-rigidity} we get
\begin{equation}\label{e-wrk-gamma}
|\gamma_o'(z)| = |\tau_o(\gamma(z))| \cdot |\gamma'(z)| 
\le |\tau(\gamma(z))| \cdot |\gamma'(z)|
\le \frac{1}{c} \tau_d(\gamma(z)) \cdot |\gamma'(z)|
= \frac{1}{c} \gamma_d'(z). 
\end{equation}
For $\eta$-a.e. $\gamma \in \Gamma_T$ there exists $t\in(0,1)$ such that $\gamma(t) \in T$.
Then by inequality \eqref{e-wrk-gamma} we obtain
\begin{equation*}
|\gamma_o(0)| \le |\gamma_o(t)| + \left|\int_0^t \gamma_o'(z) \, dz\right| \le |\gamma_o(t)| + \frac{1}{c} (\gamma_d(t) - \gamma_d(0)),
\end{equation*}
hence $\gamma(0) \in \overline T$. Clearly $\gamma_d'\ge 0$ a.e., so $\gamma_d(0) \le \gamma_d(t) \le h$, since $\gamma(t)\in T$.

Note that for $\eta$-a.e. $\gamma\in \Gamma_T$ we have $\gamma(1)\ne \gamma(0)$.
Indeed, otherwise the measure
\begin{equation*}
\bm \sigma := \int_{\{\gamma\in \Gamma : \gamma(0)=\gamma(1)\}} \bm \mu_\gamma \, d\eta(\gamma)
\end{equation*}
would be a nonzero cycle of $\bm \mu$, which is not possible since $\bm \mu$ is acyclic.
Therefore for $\eta$-a.e. $\gamma\in \Gamma_T$
\begin{equation*}
|\dive \bm \mu_\gamma|(\overline T) = \delta_{\gamma(0)}(\overline T) + \delta_{\gamma(1)}(\overline T) \ge 1
\end{equation*}
since $\gamma(0) \in \overline T$. But $|\dive \bm \mu|$ is concentrated on $A$ and $A \cap T = \emptyset$, hence
\begin{equation*}
\eta(\Gamma_T) = \int_{\Gamma_T} 1 \, d\eta(\gamma) \le \int_{\Gamma_T} |\dive \bm \mu_\gamma|(\overline T) \, d\eta(\gamma) = |\dive \bm \mu|(\overline T) = 0
\end{equation*}
and therefore $|\bm \mu|(T) = \int_{\Gamma_T} |\bm \mu_\gamma|(T) \, d\eta(\gamma) = 0$.
By arbitrariness of $h$ and $r$ we conclude that $\bm \mu = 0$.
\end{proof}

\begin{figure}[!ht]
	\centering
	\begin{tikzpicture}[scale=.9, font=\footnotesize]
	\usetikzlibrary{arrows}
	\usetikzlibrary{decorations.pathreplacing,decorations.markings}
	\tikzset{arrow data/.style 2 args={%
	      decoration={%
	         markings,
	         mark=at position #1 with \arrow{#2}},
	         postaction=decorate}
	      }%
	{\draw 
		plot coordinates {(-5,0)
			(0, 4)
			(5,0)};
	}
	
	\draw  (-2.5,0) -- (-2.5,2); 
	
	{\draw [fill=yellow!15]
		plot coordinates {(-5,0)
			(-2.5, 2)
			(2.5, 2)
			(5,0)};
	}

	\draw (-8,0) node[circle,fill,inner sep=1pt,label=below:$-R$]{}; 
	\draw (-5,0) node[circle,fill,inner sep=1pt,label=below:$-r-\frac{h}{c}$]{};
	\draw (-2.5,0) node[circle,fill,inner sep=1pt,label=below:$-r$]{};
	\draw (8,0) node[circle,fill,inner sep=1pt,label=below:$R$]{}; 
	\draw (5,0) node[circle,fill,inner sep=1pt,label=below:$r+\frac{h}{c}$]{};
	\draw (2.5,0) node[circle,fill,inner sep=1pt,label=below:$+r$]{};
	\draw (0,4) node[circle,fill,inner sep=1pt,label=right:$h+rc$]{};
	\draw (0,2) node[circle,fill,inner sep=1pt,label=above:$\,\,\,\,\,\,\,\,\,\,h$]{};

	\draw [<->] (0,5.5) node [left] {$x_d$} -- (0,-0.5) -- (0,0) -- (-9,0) -- (9,0) node [below] {$x_o$};

	\draw [black, arrow data={0.5}{stealth}] plot [smooth, tension=1] coordinates {(-1.4,0) (-1.7,0.4) (-1.5,1.0) (-1.9,1.6)};
	\draw (-1.9,1.6) node[circle,fill,inner sep=1pt,label=right:$\gamma(t)$]{};
	\draw (-1.4,0) node[circle,fill,inner sep=1pt,label=above right:$\gamma(0)$]{};
	\end{tikzpicture}
	\caption{The region depicted in yellow is the set $T$ defined in \eqref{eq:def_T} (in the case $d=2$). In the proof of Thm. \ref{t-linear-rigidity}, we show that $\eta(\Gamma_T) = 0$, i.e. the set of curves $\gamma$ such that $\bm \mu_\gamma(T)>0$ is $\eta$-negligible. From this we deduce that $\bm \mu(T) = 0$.}
	\label{fig:solution}
\end{figure}
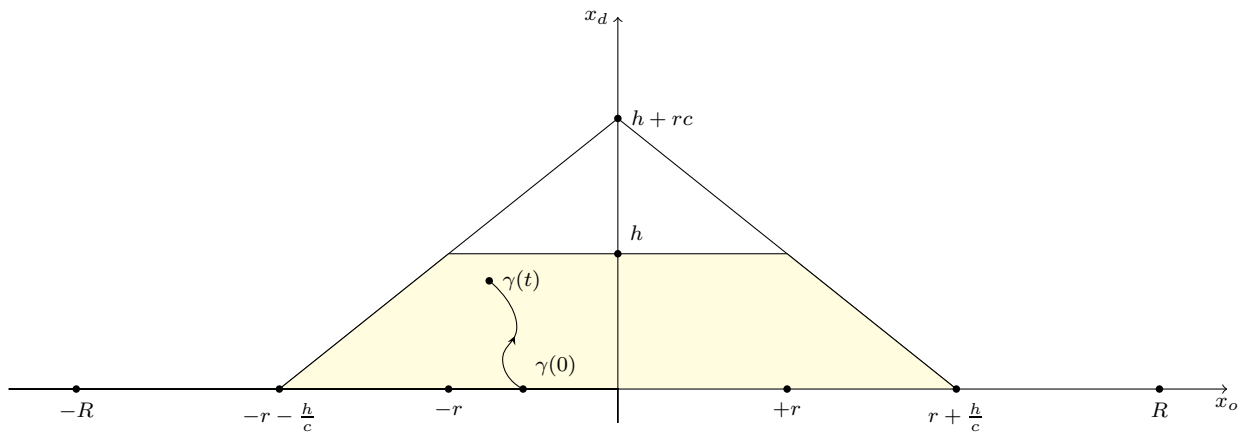

\appendix

\section{Decomposition Theorem for \texorpdfstring{$\FV$}{FV} functions}\label{app} 

We begin with the following definition. 

\begin{definition}
	A function $f \in \FV(\R^d)$ is said to be \emph{monotone} if the sets $\{f>t\}$ and $\{f\le t\}$ are indecomposable for a.e. $t \in \R$. 
\end{definition}

Notice that, by Remark \ref{rmk:still_holds}, a function $f$ such that the superlevel sets $\{f> t\}$ are \emph{simple} for a.e. $t \in \R$ is necessarily monotone.

The goal of this appendix is to give a self-contained proof of the following theorem (see also \cite{BT}).

\begin{theorem}\label{thm:BT}
	For any $f \in \FV(\R^d)$ there exists an at most countable family $\{f_i\}_{i \in I} \subset \FV(\R^d)$ of monotone functions such that
	\begin{equation}\label{eq:tesi_BT}
	f = \sum_{i\in I} f_i \qquad\text{and}\qquad |Df| = \sum_{i\in I} |Df_i|.
	\end{equation}
	In particular,
	\begin{equation*}
	\|f\|_{\FV} = \sum_{i\in I} \|f_i\|_{\FV}.
	\end{equation*}
\end{theorem}

\begin{remark}
	Observe that from the embeddings of $\FV$ (see Thm. \ref{thm:347AFP}) the first series in \eqref{eq:tesi_BT} converges also in $L^{1*}(\R^d)$ but, in general, we cannot improve this to convergence in $L^1(\R^d)$. Secondly, we remark that the decomposition provided in Theorem \ref{thm:BT} is not unique: we refer the reader to the counterexample presented in the paper \cite{BT}. 
\end{remark}

The proof of Theorem \ref{thm:BT} will be presented at the end of the appendix and it requires some preliminary lemmas. 

\begin{lemma}\label{lemma:coarea1}
	Let $\varphi,\psi \in \FV(\R^d)$ and assume $0\le \psi \le \varphi$. 
	\begin{enumerate}
		\item If for a.e. $t \in \R$ it holds 
	\begin{equation}\label{eq:assum}
	P(\{\varphi >t \}) = P( \{ \varphi>t \} \setminus \{\psi >t\}) + P ( \{ \psi>t \} )
	\end{equation}	
	then 
	\begin{equation*}
	\| \varphi \|_{\FV} = \| \varphi-\psi \|_{\FV} + \|\psi \|_{\FV}. 
	\end{equation*}
			\item If for a.e. $t \in \R$ it holds 
			
			%% f -> g e invece f+g -> f 
	\begin{equation}\label{eq:assum_prime}
	P(\{\psi >t \}) = P( \{ \varphi >t \} \setminus \{ \psi>t\}) + P ( \{ \varphi>t \} )
	\end{equation}	
	then 
	\begin{equation*}
	\| \psi \|_{\FV} = \| \varphi \|_{\FV} + \|\varphi-\psi \|_{\FV}. 
	\end{equation*}
\end{enumerate}
\end{lemma}

\begin{proof} We present the proof of the two claims. 
	\begin{enumerate}
		\item Concerning the first point, it suffices to show
	\begin{equation*}
	\| D\varphi \|_{\mathscr M} \ge \| D(\varphi-\psi) \|_{\mathscr M} + \|D\psi \|_{\mathscr M},
	\end{equation*} 
	because the other inequality is trivial by triangle inequality. Using the layer cake representation and Fubini's Theorem we get 
	\begin{equation*}
	\begin{split}
	\| D(\varphi-\psi) \|_{\mathscr M} & = \sup_{\Vert \omega \Vert_{\infty} \le 1} \int_{\R^d} (\varphi(x)-\psi(x)) \cdot \dive \omega(x) \, dx \\
	& = \sup_{\Vert \omega \Vert_{\infty} \le 1} \int_{\R^d} \int_0^\infty \left( \1_{ \{\varphi>t\} }(x) - \1_{ \{\psi >t\} }(x) \right)\cdot \dive \omega(x) \, dt \,  dx \\	
	& = \sup_{\Vert \omega \Vert_{\infty} \le 1} \int_{\R^d} \int_0^\infty \1_{ \{\varphi>t\} \setminus \{\psi >t\} }(x) \cdot \dive \omega(x) \, dt \,  dx \\	
	& = \sup_{\Vert \omega \Vert_{\infty} \le 1}  \int_0^\infty \int_{\R^d} \1_{ \{\varphi>t\} \setminus \{\psi >t\} }(x) \cdot \dive \omega(x) \, dx \,dt  \\
	& \le \sup_{\Vert \omega \Vert_{\infty} \le 1}  \int_0^\infty \langle D  \1_{ \{\varphi>t\} \setminus \{\psi >t\} }, \omega \rangle \, dt \\
	& \le \int_0^\infty P( \{\varphi>t\} \setminus \{\psi >t\} ) \, dt \\
	& \overset{\eqref{eq:assum}}{ = }   \int_0^\infty P( \{\varphi>t\}) - P( \{\psi >t\} ) \, dt. \\
	\end{split}
	\end{equation*}
	Applying again Coarea formula we obtain the conclusion. 
	\item The proof of the second claim is similar to the proof of the first one. Notice that $|Dw| =|D(-w)|$ as measures for any $w \in \FV(\R^d)$ hence 
	\begin{equation*}
	\|D\psi\|_{\mathscr M} \le  \|D\varphi\|_{\mathscr M}  + \|D(\varphi-\psi)\|_{\mathscr M}, 
	\end{equation*}
	which is equivalent to 
	\begin{equation*}
	\|D(\varphi-\psi)\|_{\mathscr M} \ge \|D\psi\|_{\mathscr M} - \|D\varphi\|_{\mathscr M}. 
	\end{equation*}
	It thus remains to show 
	\begin{equation*}
	\|D(\varphi-\psi)\|_{\mathscr M} \le \|D\psi \|_{\mathscr M}  - \|D\varphi\|_{\mathscr M}. 
	\end{equation*}
	By layer cake representation and Fubini, as in Point (1), we have 
	\begin{equation*}
	\begin{split}
	\| D(\varphi -\psi) \|_{\mathscr M} & = \sup_{\Vert \omega \Vert_{\infty} \le 1} \int_{\R^d} (\varphi(x)-\psi(x)) \cdot \dive \omega(x) \, dx \\
	& = \sup_{\Vert \omega \Vert_{\infty} \le 1} \int_{\R^d} \int_0^\infty \left( \1_{ \{\varphi>t\} }(x) - \1_{ \{\psi >t\} }(x) \right)\cdot \dive \omega(x) \, dt \,  dx \\	
	& = \sup_{\Vert \omega \Vert_{\infty} \le 1} \int_{\R^d} \int_0^\infty \1_{ \{\varphi>t\} \setminus \{\psi >t\} }(x) \cdot \dive \omega(x) \, dt \,  dx \\	
	& = \sup_{\Vert \omega \Vert_{\infty} \le 1}  \int_0^\infty \int_{\R^d} \1_{ \{\varphi>t\} \setminus \{\psi>t\} }(x) \cdot \dive \omega(x) \, dx \,dt  \\
	& \le \sup_{\Vert \omega \Vert_{\infty} \le 1}  \int_0^\infty \langle D  \1_{ \{\varphi>t\} \setminus \{\psi >t\} }, \omega \rangle \, dt \\
	& \le \int_0^\infty P( \{\varphi>t\} \setminus \{\psi>t\} ) \, dt \\
	& \overset{\eqref{eq:assum_prime}}{ = }   \int_0^\infty P( \{\psi>t\}) - P( \{\varphi >t\} ) \, dt. \\
	\end{split}
	\end{equation*}
	Again the application of Coarea Formula yields the desired conclusion.
\end{enumerate}
\end{proof}

\begin{lemma}[From superlevel sets to function]\label{lemma:construction_function_from_superlevels} 	
Let $I \subset [0,+\infty)$ be an interval and let $(A_t)_{t \in I}$ be a family of sets such that $t,s \in I$ with $s<t$ implies $A_t \subset A_s$. Then there exists a measurable function $w \colon \R^d \to [0,+\infty)$ such that $\{w>t\} = A_t$ (up to Lebesgue negligible subsets) for a.e. $t \in I$.
\end{lemma}

\begin{proof} Due to monotonicity of the family $(A_t)_{t \in I}$, the function $h(t) := \meas{A_t}$ is non-increasing on $I$. Therefore there exists a Lebesgue negligible set $N \subset I$ such that
$h$ is continuous at every $t\in I \setminus N$. Let $Q \subseteq I \setminus N$ be a countable set, which is dense in $I$. For any $x\in \R^d$ we define
\begin{equation*}
w(x) := \sup_{t \in Q} \bigl( t\cdot  \1_{A_t}(x)\bigr).
\end{equation*}
Clearly $w$ is Lebesgue measurable. By definition of $w$ for any $s\in I \setminus N$
\begin{equation*}
\{w > s\} = \bigcup_{t\in Q \cap (s,+\infty) \cap I } A_t
\end{equation*}
Since for any $s<t$ it holds $\meas{A_t \setminus A_s} = 0$ and $Q$ is countable, it follows that
\begin{equation*}
\meas{\Biggl(\bigcup_{t\in Q \cap (s,+\infty) \cap I } A_t\Biggr) \setminus A_s} = 0.
\end{equation*}
On the other hand, let $\eps:=\meas{A_s \setminus \bigcup_{t\in Q \cap (s,+\infty) \cap I} A_t}$.
For any $t\in Q \cap (s,+\infty) \cap I $ we have $A_t \subset A_s$, hence $\bigcup_{t\in Q \cap (s,+\infty) \cap I} A_t \subset A_s$. In particular, we can estimate 
\begin{equation*}
\meas{A_s} = \meas{ A_s \setminus \bigcup_{t\in Q \cap (s,+\infty) \cap I} A_t} + \meas{ \bigcup_{t\in Q \cap (s,+\infty) \cap I} A_t} \ge \e + \meas{A_t}. 
\end{equation*}
Since $Q$ is dense in $(s,+\infty) \cap I $ and $h$ is continuous at $s$ the only possible case is $\eps=0$. We have thus proved that $\meas{\{g>s\} \mathbin{\triangle} A_s}=0$ for a.e. $s \in I$ and this concludes the proof. 
\end{proof}

The following lemma is a building block of the proof of Theorem~\ref{thm:BT}.
It allows to ``extract'' from a non-negative $\FV$ function (whose superlevel sets in general are not indecomposable) a non-trivial function with indecomposable superlevel sets:

% Decomposition into indecomposable super-level sets
\begin{lemma}[Extraction lemma I] \label{lemma:extraction} Let $f \in \FV(\R^d)$ and assume $f$ is not identically zero and non-negative. Then there exists $g \in \FV(\R^d)$ with $0\le g \le f$ and $g \not\equiv 0$ such that: 
	\begin{enumerate}
		\item[(i)] for a.e. $t \ge 0$ the set $\{g > t\}$ is indecomposable; 
		\item[(ii)] it holds $\Vert f \Vert_{\FV} =\Vert f -g \Vert_{\FV} + \Vert g \Vert_{\FV}$. 
	\end{enumerate}
\end{lemma}

\begin{proof}
For any $t\ge 0$ let $E_t := \{f > t\}$.
Since $f\in \FV(\R^d)$, there exists a Lebesgue negligible set $N \subseteq (0,+\infty)$ such that for any $t \in (0,+\infty) \setminus N$ the set $E_t$ has finite perimeter.
Let $E_t^{k}$ denote the $k$-th $M$-connected component of $E_t$, 
$t \in (0,+\infty) \setminus N$.

Fix some $a>0$ such that $\meas{E_a}>0$.
Let $R$ be some $M$-connected component of $E_a$.
For any $t\in(0,a) \setminus N$ we have $E_t \supseteq E_a \supseteq R$,
and $R$ is indecomposable, hence by Theorem~\ref{thm:ACMM} there exists a unique $j=j(t)$ such that $\meas{R \setminus E_t^{j(t)}} = 0$.

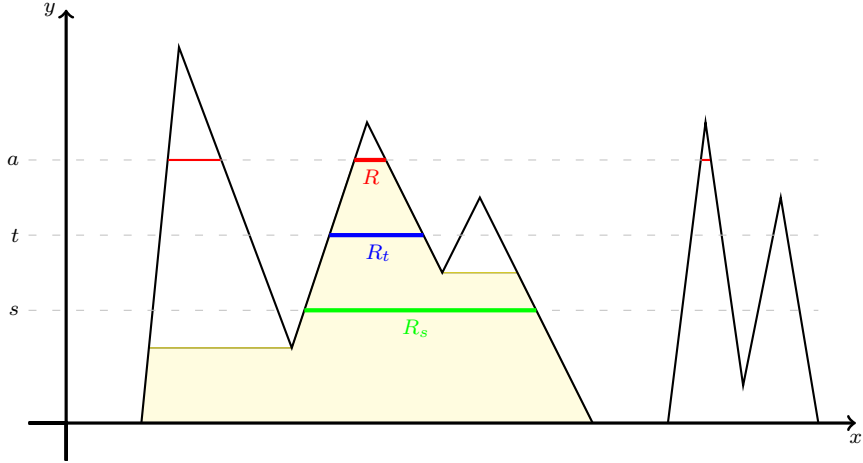
\begin{figure}[!ht]
	\centering
	\begin{tikzpicture}[scale=1, font=\footnotesize]

	{\draw [ultra thin, black, fill=yellow!15]
		plot coordinates {(1.1,1)
			(3,1)
			(3.83, 3.5)
			(4.25, 3.5)
			(5,2)
			(6,2)
			(7,0)
			(1,0)
		};
	}
	
	{\draw[ultra thin,yellow!90]
		plot coordinates {(1.1,1)
			(3,1)
			(3.83, 3.5)
			(4.25, 3.5)
			(5,2)
			(6,2)
			(7,0)
			(1,0)
		};
	}

	%%AXES 
	\draw [very thick, <->] (0,5.5) node [left] {$y$} -- (0,-0.5) -- (0,0) -- (-0.5,0) -- (10.5,0) node [below] {$x$};
	%%% FIRST PIECE
	{\draw[thick, black]
		plot coordinates {(1,0)
			(1.5, 5)
			(3,1)
			(4,4)
			(5,2)
			(5.5,3)
			(7,0)
		};
	}
	
	%%%% SECOND PIECE 
	{\draw[thick,black]
		plot coordinates {(8,0)
			(8.5,4)
			(9,0.5)
			(9.5,3)
			(10,0)
		};
	}

	\draw[loosely dashed, very thin, gray!50]  (-0.5,3.5) -- (10,3.5); 
	\draw[loosely dashed, very thin, gray!50]  (-0.5,2.5) -- (10,2.5);
	\draw[loosely dashed, very thin, gray!50]  (-0.5,1.5) -- (10,1.5);
	
	\draw[red, ultra thick] (3.83, 3.5) -- (4.25, 3.5);
	\draw[blue, ultra thick] (3.5, 2.5) -- (4.75, 2.5);
	\draw[green, ultra thick] (3.16667, 1.5) -- (6.25, 1.5);
	
	\draw[red, thick] (8.45, 3.5) -- (8.56, 3.5);
	\draw[red, thick] (1.35555556, 3.5) -- (2.05, 3.5);
	
	\draw (-0.5, 3.5) node [left] {$a$};
	\draw (-0.5, 2.5) node [left] {$t$};
	\draw (-0.5, 1.5) node [left] {$s$};
	
	\draw (4.05, 3.5) node [below] { \textcolor{red}{$R$} };
	
	\draw (4.15, 2.5) node [below] { \textcolor{blue}{$R_t$} };
	\draw (4.65, 1.5) node [below] { \textcolor{green}{$R_s$} };
	
	\end{tikzpicture}
	\caption{Situation described in the proof of Lemma \ref{lemma:extraction}. The black curve represents the graph of a generic function $f \in \FV(\R^d)$. The red segments make up the level set $E_a$. The red, thick segment is the component $R$ and the blue and green ones are respectively $R_t$ and $R_s$. The area depicted in yellow is the subgraph of the function $g$, whose superlevel sets are indecomposable.}
	\label{fig:proof}
\end{figure}

Let $R_t := E_t^{j(t)}$, $t\in(0,a) \setminus N$.
Note that for any $s,t \in (0,a) \setminus N$ with $s<t$ it holds that 
\begin{equation}\label{e-R_t-subset-of-R_s}
\meas{R_t \setminus R_s}=0.
\end{equation}
Indeed, $E_s \supseteq E_t \supseteq R_t$ and $R_t$ is indecomposable, hence again by Theorem~\ref{thm:ACMM} there exists a unique $k$ such that $\meas{R_t \setminus E_s^{k}} = 0$.
But $\meas{R \setminus R_t} = 0$, hence 
\begin{equation*}
\begin{aligned}
E_s^k \setminus R &= (E_s^k \cap R_t \cap R^c) \cup (E_s^k \cap R_t^c \cap R^c) \\
&\subseteq (R_t \setminus R) \cup (E_s^k \setminus R_t)
\end{aligned}
\end{equation*}
is Lebesgue negligible. Therefore $k=j(s)$ by uniqueness of $j(s)$. Applying now Lemma \ref{lemma:construction_function_from_superlevels}, we can construct a function $g\colon \R^d \to [0,a]$ such that $\{g>s\} = R_s$ (up to Lebesgue negligible subsets) for a.e. $s \in (0,a)$.

Observe that $\|f\|_{\FV} = \|\bar f\|_{\FV} + \|\hat f\|_{\FV}$, where $\bar f(x) = \min(a, f(x))$
and $\hat f := f - \bar f$.
For a.e. $t\in (0,a)$ we have
\begin{equation*}
\{\bar f>t\} = \{f>t\} = E_t^{j(t)}  \cup \bigcup_{k\ne j(t)} E_t^{k}
\end{equation*}
hence by construction of $g$ and Proposition~\ref{p-decomposition-th+}
\begin{equation*}
P(\{\bar f>t\}) = P ( \{ g>t \} ) + P( \{ f>t \} \setminus \{g >t\}). 
\end{equation*}
Hence by Lemma~\ref{lemma:coarea1} %and the additivity of the Lebesgue integral 
we have
\begin{equation}
\|\bar f\|_{\FV} = \|g\|_{\FV} + \|\bar f - g\|_{\FV}.
\end{equation}
Then by the triangle inequality %(denoting $\|\cdot\|:=\|\cdot\|_{\BV}$ for brevity)
\begin{equation*}
\|f\|_{\FV} = \|\bar f\|_{\FV}+ \|\hat f\|_{\FV} = \|g\|_{\FV} + \|\bar f - g\|_{\FV} + \|\hat f\|_{\FV} \ge \|g\|_{\FV}+ \|\bar f - g + \hat f\|_{\FV}
\end{equation*}
and $\|f\|_{\FV} = \|g + \bar f - g + \hat f\|_{\FV} \le \|g\|_{\FV} + \|\bar f - g + \hat f\|_{\FV}$,
hence the property \textit{(ii)} follows.
\end{proof}

% Decomposition into SIMPLE super-level sets
\begin{lemma}[Extraction lemma II] \label{lemma:extraction_plus} Let $f \in \FV(\R^d)$ and assume $f$ is not identically zero and non-negative. Then there exists $h \in \FV(\R^d)$ with $h \not\equiv 0$ such that: 
	\begin{enumerate}
		\item[(i)] for a.e. $t \ge 0$ the set $\{h > t\}$ is simple;
		\item[(ii)] it holds $\Vert f \Vert_{\FV} =\Vert f -h \Vert_{\FV} + \Vert h \Vert_{\FV}$. 
	\end{enumerate}
\end{lemma}

\begin{proof}
	First of all, we apply Lemma \ref{lemma:extraction} and we obtain a function $g \in \FV(\R^d)$ such that $G_t:=\{g>t\}$ is indecomposable for a.e. $t \ge 0$ and 
	\begin{equation}\label{eq:indec_first_step}
	\Vert f \Vert_{\FV} =\Vert f -g \Vert_{\FV} + \Vert g \Vert_{\FV}. 
	\end{equation}
	Let us now work on the function $g$. By construction of $g$, for a.e. $t \ge 0$ the set $G_t$ is indecomposable. Fix some $a>0$ such that $\meas{G_a}>0$ and $G_a$ is not simple (otherwise there is nothing to prove): let us denote by $\{F_t^i\}_{i\in I_t}$ the non-empty family of holes of $G_t$ (i.e. $\mathcal{CC}^M(\R^d \setminus G_t) = \{F_t^i\}_{i\in I_t}$). 
	
	Observe that, if $H$ is an hole of $G_a$, for any $t\in(a,+\infty) \setminus N$, we have $G_t \subseteq G_a$, and hence $G_t^c \supseteq G_a^c \supseteq H$: this means that $H$ is an hole of $G_t$ for any $t\in(a,+\infty) \setminus N$: by the uniqueness claim in Theorem~\ref{thm:ACMM} there exists a unique $j=j(t)$ such that $\meas{H \setminus F_t^{j(t)}} = 0$.
	
	For any $t \in (0,a)$ define $S_t := \text{sat} (G_t)$. Observe that the sequence $(S_t)_{ t \in (0,a)}$ is monotone \cite[Prop. 6(iii)]{ACMM} and thus, applying Lemma \ref{lemma:construction_function_from_superlevels}, we obtain a function $h \colon \R^d \to \R$ such that $\{h>r\}= S_r$ (up to Lebesgue negligible subsets) for a.e. $r \in (0,a)$. By construction the function $h$ is non-negative and $\{h>r\}$ is simple for a.e. $r \in (0,a)$, because the saturation of an indecomposable set is simple. It thus remains to show property \textit{(ii)} of the statement. For, 
	notice preliminarly, that $h-g \ge 0$ by construction of $h$; by \cite[Prop. 9]{ACMM}, it holds for any $t\in(a,+\infty) \setminus N$ 
	\begin{equation*}
	P(G_t) = P(\text{sat} (G_t)) + P\left (\bigcup_{i \in I_t} F_t^i \right) 
	\end{equation*}
	which can be also written as 
	\begin{equation*}
	P(\{g>t\}) = P(\{h>t\}) + P\left ( \{h>t\} \setminus \{g>t\} \right).
	\end{equation*}
	We are now in position to apply Lemma \ref{lemma:coarea1}, Point (ii), choosing $\varphi := h$ and $\psi := g$ (which is possible since $h \ge g$): we obtain 
	\begin{equation}\label{eq:indec_second_step}
	\| g \|_{\FV} = \| h \|_{\FV} + \|g-h \|_{\FV}. 
	\end{equation}
	It is now easy to check that property \textit{(ii)} follows combining \eqref{eq:indec_first_step} with \eqref{eq:indec_second_step} - and the triangle inequality: 
	\begin{equation*}
	\begin{split}
	\Vert f \Vert_{\FV} & \le \Vert f -h \Vert_{\FV} + \Vert h \Vert_{\FV} \\
	& \le \Vert f -g \Vert_{\FV} + \Vert g-h \Vert_{\FV} + \Vert h \Vert_{\FV} \\
	& \overset{\eqref{eq:indec_first_step}}{ = }  \Vert f \Vert_{\FV} - \|g \|_{\FV}+ \Vert g-h \Vert_{\FV} + \Vert h \Vert_{\FV} \\
	& \overset{\eqref{eq:indec_second_step}}{ = }  \Vert f \Vert_{\FV} - \|g \|_{\FV}+ \Vert g \Vert_{\FV}- \Vert h \Vert_{\FV} + \Vert h \Vert_{\FV}  = 	\Vert f \Vert_{\FV}
	\end{split}
	\end{equation*}
	and this completes the proof. 	
\end{proof}

\begin{lemma}[Extraction lemma III] \label{lemma:extraction_plus_plus} Let $f \in \FV(\R^d)$ and assume $f$ is not identically zero. Then there exists $m \in \FV(\R^d)$ with $m \not\equiv 0$ such that: 
	\begin{enumerate}
		\item[(i)] $m$ is monotone and $\sign m = \text{constant}$ a.e.; 
		\item[(ii)] it holds $\Vert f \Vert_{\FV} =\Vert f -m \Vert_{\FV} + \Vert m \Vert_{\FV}$. 
	\end{enumerate}
\end{lemma}

\begin{proof}
	Let us decompose $f = f^+ - f^-$. Suppose $\|f^+\|_{\FV}>0$. Since $f^+ \ge 0$ we can apply Lemma \ref{lemma:extraction} to $f^+$, thus obtaining a function $u \ge 0$ such that $\{u > t\}$ is indecomposable for a.e. $t >0$ and it holds 
	\begin{equation}\label{eq:mono1}
	\Vert f^+ \Vert_{\FV} =\Vert f^+ -u \Vert_{\FV} + \Vert u \Vert_{\FV}.
	\end{equation} 
	Applying now Lemma \ref{lemma:extraction_plus} to $u \ge 0$ we obtain a function $m \in \FV(\R^d)$ such that for a.e. $t \ge 0$ the set $\{m > t\}$ is simple and it holds
	\begin{equation}\label{eq:mono2}
	\Vert u \Vert_{\FV} =\Vert u -m \Vert_{\FV} + \Vert m \Vert_{\FV}.
	\end{equation}
	By triangle inequality 
	\begin{equation*}
	\begin{split}
	\Vert f \Vert_{\FV} & \le \Vert f  - m  \|_{\FV} + \| m \|_{\FV} \\ 
	& \le \Vert f^+ - m  \|_{\FV} + \| f^- \|_{\FV} + \| m \|_{\FV} \\ 
	& \le \| f^+ - u  \|_{\FV} + \| u - m  \|_{\FV} + \| f^- \|_{\FV} + \| m \|_{\FV} \\ 
	&  \overset{\eqref{eq:mono2}}{ = } \| f^+ - u  \|_{\FV} + \| u  \|_{\FV} + \| f^- \|_{\FV} \\
	&  \overset{\eqref{eq:mono1}}{ = } \| f^+ \|_{\FV} + \| f^- \|_{\FV}  = \Vert f \Vert_{\FV}
	\end{split}
	\end{equation*}
	hence Property (ii) holds true. Since the function $m$ is monotone, this concludes the proof in the case $\| f^+\|_{\FV} > 0$.
	% (observe that the function $m$ is indeed monotone: it is always true that a simple set is indecomposable and has indecomposable complement, independently on the fact that its measure is finite or not).
	It remains to consider the case in which $f^+ \equiv 0$. If $f^- \equiv 0$ there is nothing to prove; if $\| f^- \|_{\FV} >0$ then we repeat the same argument above for the function $\widetilde{f}:=-f \in \FV(\R^d)$. We end up with a monotone function $\widetilde{m}$ of constant sign such that 
	\begin{equation*}
	\Vert \widetilde f \Vert_{\FV} =\Vert \widetilde f -\widetilde m \Vert_{\FV} + \Vert \widetilde m \Vert_{\FV}
	\end{equation*}
	which is clearly equivalent to Property (ii) (renaming $-\widetilde m$ as $m$). 
	
\end{proof}

Now we prove Theorem \ref{thm:BT} using Lemma~\ref{lemma:extraction_plus_plus} and transfinite induction:

\begin{proof}[Proof of Theorem \ref{thm:BT}]
Let $X:=\{g\in \FV(\R^d) : \text{$g$ is monotone and $\|g\|_{\FV} > 0$}\}$.
For any $h \in \FV(\R^d)$ let $Y(h) := \{g \in X : \|h\|_{\FV} = \|h-g\|_{\FV} + \|g\|_{\FV}\}$.
Note that by Lemma~\ref{lemma:extraction_plus_plus} $Y(h) = \emptyset$ if and only if $h\equiv 0$.
Ultimately let $\mathfrak{s}\colon \mathscr{P}(\FV(\R^d)) \to \FV(\R^d)$ denote a choice function (given by Axiom of Choice).

For any ordinal $\alpha < \omega_1$ (where $\omega_1$ is the first uncountable ordinal) and any transfinite sequence $\{g_\xi\}_{\xi < \alpha} \subset X \cup \{{{\infty}}\}$ let us define
\begin{equation*}
E(\{g_\xi\}_{\xi<\alpha}) := \begin{cases}
{{\infty}}, & \text{if ${{\infty}} \in \{g_\xi\}_{\xi < \alpha}$ or if $\sum_{\xi < \alpha} \|g_\xi\|_{\FV} = \infty$}; \\
\mathfrak{s}(Y(f - \sum_{\xi < \alpha} g_\xi)), & \text{if $\sum_{\xi < \alpha} \|g_\xi\|_{\FV} < \infty$ and $Y(f - \sum_{\xi < \alpha} g_\xi) \ne \emptyset$;} \\
0, & \text{if $\sum_{\xi < \alpha} \|g_\xi\|_{\FV} < \infty$ and $Y(f - \sum_{\xi < \alpha} g_\xi) = \emptyset$}.
\end{cases}
\end{equation*}
By transfinite recursion (see e.g. \cite[p. 21]{jech}) there exists a transfinite sequence $\{g_\alpha\}_{\alpha<\omega_1}$ such that $g_\alpha = E(\{g_\xi\}_{\xi < \alpha})$ for any $\alpha < \omega_1$.

Note that for any $\alpha < \omega_1$ the following properties hold:
\begin{subequations}
\begin{equation}\label{e-transfin-1}
{{\infty}} \notin \{g_\xi\}_{\xi < \alpha},
\end{equation}
\begin{equation}\label{e-transfin-2}
\sum_{\xi < \alpha} \|g_\xi\|_{\FV} \le \|f\|_{\FV},
\end{equation}
\begin{equation}\label{e-transfin-3}
\|f\|_{\FV} = \|f - \sum_{\xi < \alpha} g_\xi \|_{\FV} + \sum_{\xi < \alpha} \|g_\xi\|_{\FV}.
\end{equation}
\end{subequations}
Observe that \eqref{e-transfin-2} follows from \eqref{e-transfin-3}, but without \eqref{e-transfin-2} the term $\sum_{\xi < \alpha} g_\xi$ in \eqref{e-transfin-1} is not well-defined. Indeed, these properties trivially hold for $\alpha = 0$. Let $\beta <\omega_1$ and suppose that these properties hold for any $\alpha < \beta$. In order to show that \eqref{e-transfin-1}--\eqref{e-transfin-3} hold with $\alpha=\beta$ we consider two cases.

First, if $\beta$ is not a limit ordinal, then $\beta = \gamma + 1$ for some ordinal $\gamma$, so by definition of $\{g_\xi\}_{\xi < \omega_1}$ we have $g_{\gamma+1} = \mathfrak{s}(Y(f - \sum_{\xi < \gamma} g_\xi))$. Hence
\begin{equation*}
\|f\|_{\FV}=\|f-\sum_{\xi<\gamma} g_\xi\|_{\FV} + \sum_{\xi<\gamma}\|g_\xi\|_{\FV}
= \|f-\sum_{\xi<\gamma} g_\xi - g_{\gamma}\|_{\FV} + \|g_{\gamma}\|_{\FV} + \sum_{\xi<\gamma}\|g_\xi\|_{\FV}
\end{equation*}
and it follows that \eqref{e-transfin-1}--\eqref{e-transfin-3} hold with $\alpha=\gamma+1$.

Second, if $\beta$ is a limit ordinal then $\beta = \bigcup_{\alpha < \beta} \alpha$. Consequently
$\{g_\xi\}_{\xi < \beta} = \bigcup_{\alpha < \beta} \{g_\xi\}_{\xi < \alpha}$, hence \eqref{e-transfin-1} holds with $\alpha = \beta$. Furthermore, since $\beta$ is at most countable we can enumerate it as $\beta = \{\alpha_n\}_{n\in \N}$. Let $A_n := \alpha_1 \cup \ldots \cup \alpha_n$ (note that for any $n\in \N$ there exists $m\in \{1,\ldots, n\}$ such that $A_n = \alpha_m$). Since $\beta = \bigcup_{\alpha < \beta} \alpha = \bigcup_{n\in \N} A_n$ we have
\begin{equation*}
\sum_{\xi < \beta} \|g_{\xi}\|_{\FV} = \sum_{\xi < \beta} (\sup_{n\in\N} \1_{A_n}(\xi)) \|g_{\xi}\|_{\FV} = \sup_{n\in \N} \sum_{\xi \in A_n} \|g_{\xi}\|_{\FV} \le \sup_{\alpha < \beta} \|g_{\xi}\|_{\FV} \le \|f\|_{\FV},
\end{equation*}
hence \eqref{e-transfin-2} holds with $\alpha = \beta$. Consequently $\sum_{\xi < \beta} g_\xi = \lim_{n\to\infty} \sum_{\xi \in A_n} g_\xi$ and $\sum_{\xi < \beta} \|g_{\xi}\|_{\FV} = \lim_{n\to\infty} \sum_{\xi \in A_n} \|g_{\xi}\|_{\FV}$. Writing \eqref{e-transfin-3} with $\alpha = A_n$ and passing to the limit as $n\to \infty$ we conclude that \eqref{e-transfin-3} holds with $\alpha = \beta$.

We have thus shown that \eqref{e-transfin-1}--\eqref{e-transfin-3} hold with $\alpha=\beta$.
Hence by transfinite induction \eqref{e-transfin-1}--\eqref{e-transfin-3} hold for any $\alpha < \omega_1$.

By \eqref{e-transfin-2} for any $\e >0$ the set 
\begin{equation*}
\{\alpha< \omega_1 : \Vert g_\alpha \Vert_{\FV}>\e   \}
\end{equation*}
is finite and thus the set 
\begin{equation*}
A:=\{\alpha< \omega_1 : \Vert g_\alpha \Vert_{\FV}>0 \}
\end{equation*}
is at most countable. Setting $\gamma := \sup A$ we have $g_{\gamma +1} = 0$.
As already noted above, by Lemma~\ref{lemma:extraction_plus_plus} this means that $f = \sum_{\xi<\gamma} g_\xi$, and $\|f\|_{\FV} = \sum_{\xi < \gamma} \|g_\xi\|_{\FV}$ by \eqref{e-transfin-3}.

By triangle inequality $|Df| \le \sum_{\xi < \gamma} |D g_\xi|$.
If this inequality were strict, we would have $\|f\|_{\FV} = |Df|(\R^d) < \sum_{\xi < \gamma} |D g_\xi|(\R^d) = \sum_{\xi < \gamma} \|g_\xi\|_{\FV} = \|f\|_{\FV}$, which is a contradiction.
\end{proof}

\section*{Acknowledgements}
The first author was supported by ERC Starting Grant 676675 FLIRT. He thanks Graziano Crasta and Annalisa Malusa for interesting discussions during the preparation of the paper; he also kindly acknowledges Gian Paolo Leonardi for introducing him to the problem of rigidity of divergence-free vector fields. This project has received funding from the European Research Council (ERC) under the European Union's Horizon 2020 research and innovation programme under grant agreement No 757254 (SINGULARITY).
The work of the second author was supported by RFBR Grant 18-31-00279.
	
	\bibliographystyle{alpha}
	\bibliography{biblio}

\begin{thebibliography}{ACMM01}

\bibitem[AC08]{umibologna}
L.~Ambrosio and G.~Crippa.
\newblock {\em Existence, Uniqueness, Stability and Differentiability
  Properties of the Flow Associated to Weakly Differentiable Vector Fields},
  pages 3--57.
\newblock Springer Berlin Heidelberg, Berlin, Heidelberg, 2008.

\bibitem[ACMM01]{ACMM}
L.~Ambrosio, V.~Caselles, S.~Masnou, and J.-M. Morel.
\newblock Connected components of sets of finite perimeter and applications to
  image processing.
\newblock {\em Journal of the European Mathematical Society}, 3(1):39--92, Feb
  2001.

\bibitem[AFP00]{AFP}
L.~Ambrosio, N.~Fusco, and D.~Pallara.
\newblock {\em Functions of Bounded Variation and Free Discontinuity Problems}.
\newblock Oxford Science Publications. Clarendon Press, 2000.

\bibitem[Anz83]{Anzellotti1983}
G.~Anzellotti.
\newblock Pairings between measures and bounded functions and compensated
  compactness.
\newblock {\em Annali di Matematica Pura ed Applicata}, 135(1):293--318,
  December 1983.

\bibitem[BBG16]{BBG2016}
S.~Bianchini, P.~Bonicatto, and N.~A. Gusev.
\newblock Renormalization for autonomous nearly incompressible {BV} vector
  fields in two dimensions.
\newblock {\em {SIAM} Journal on Mathematical Analysis}, 48(1):1--33, January
  2016.

\bibitem[BG16]{BG}
S.~Bianchini and N.~A. Gusev.
\newblock Steady nearly incompressible vector fields in two-dimension: Chain
  rule and renormalization.
\newblock {\em Archive for Rational Mechanics and Analysis}, 222(2):451--505,
  Nov 2016.

\bibitem[Bog06]{bogachev}
V.I. Bogachev.
\newblock {\em Measure Theory}.
\newblock Springer Berlin Heidelberg, 2006.

\bibitem[BT11]{BT}
S.~Bianchini and D.~Tonon.
\newblock {A decomposition theorem for BV functions}.
\newblock {\em {Communications on Pure and Applied Mathematics}}, 10(6):1549 --
  1566, 2011.

\bibitem[EG91]{EG}
L.C. Evans and R.F. Gariepy.
\newblock {\em Measure Theory and Fine Properties of Functions}.
\newblock Studies in Advanced Mathematics. Taylor \& Francis, 1991.

\bibitem[Hat02]{hatcher}
A.~Hatcher.
\newblock {\em Algebraic Topology}.
\newblock Algebraic Topology. Cambridge University Press, 2002.

\bibitem[Jec06]{jech}
T.~Jech.
\newblock {\em Set Theory: The Third Millennium Edition, revised and expanded}.
\newblock Springer Monographs in Mathematics. Springer Berlin Heidelberg, 2006.

\bibitem[Kol83]{koliada}
V.~I. Koliada.
\newblock On the metric {D}arboux property.
\newblock {\em Analysis Mathematica}, 9(4):291--312, December 1983.

\bibitem[LS20]{LeoSar17}
G.P Leonardi and G.~Saracco.
\newblock Rigidity and trace properties of divergence-measure vector fields.
\newblock {\em Advances in Calculus of Variations}, 2020.

\bibitem[Mag12]{Maggi}
F.~Maggi.
\newblock {\em Sets of Finite Perimeter and Geometric Variational Problems: An
  Introduction to Geometric Measure Theory}.
\newblock Cambridge Studies in Advanced Mathematics. Cambridge University
  Press, 2012.

\bibitem[Phe01]{phelps}
R.R. Phelps.
\newblock {\em Lectures on Choquet's Theorem}.
\newblock Lecture Notes in Mathematics. Springer Berlin Heidelberg, 2001.

\bibitem[PS12]{PS2012}
E.~Paolini and E.~Stepanov.
\newblock Decomposition of acyclic normal currents in a metric space.
\newblock {\em Journal of Functional Analysis}, 263(11):3358--3390, December
  2012.

\bibitem[PS13]{PS2013}
E.~Paolini and E.~Stepanov.
\newblock Structure of metric cycles and normal one-dimensional currents.
\newblock {\em Journal of Functional Analysis}, 264(6):1269--1295, March 2013.

\bibitem[Rud06]{rudin}
W.~Rudin.
\newblock {\em Functional Analysis}.
\newblock International series in pure and applied mathematics. McGraw-Hill,
  2006.

\bibitem[Smi93]{Smi93}
S.~K. Smirnov.
\newblock Decomposition of solenoidal vector charges into elementary solenoids,
  and the structure of normal one-dimensional flows.
\newblock {\em Algebra i Analiz}, 5:206--238, 1993.

\bibitem[ST17]{ST2017}
E.~Stepanov and D.~Trevisan.
\newblock Three superposition principles: Currents, continuity equations and
  curves of measures.
\newblock {\em Journal of Functional Analysis}, 272(3):1044--1103, February
  2017.

\end{thebibliography}

\end{document}